\documentclass[12pt,letterpaper]{article}
\usepackage{amssymb,amsmath,amsthm}
\usepackage{subcaption}
\usepackage{multirow}
\usepackage{graphicx}
\usepackage{booktabs}
\usepackage{sectsty}
\usepackage[numbers,square]{natbib}
\sectionfont{\normalsize}
\subsectionfont{\normalsize}
\subsubsectionfont{\normalsize}
\paragraphfont{\normalsize}
\newtheorem{thm}{Theorem}
\newtheorem*{thm*}{Theorem}
\newtheorem{lem}[thm]{Lemma}

\newtheorem{cor}[thm]{Corollary}
\newtheorem{clm}[thm]{Claim}
\newcommand{\llll}[1] {\left #1}
\newcommand{\rrrr}[1] {\right #1}

\newcommand{\dddd}[2]{\dfrac{#1}{#2}}
\newcommand{\pppp}{\partial}
\newcommand{\aaaa}{\alpha}
\newcommand{\tttt}{\tau}
\newcommand{\ssss}{\sigma}

\newcommand{\dddddd}{\delta}
\newcommand{\llllll}{\lambda}
\newcommand{\bbbb}{\beta}
\newcommand{\GGGG}{\Gamma}
\newcommand{\oooo}{\omega}
\newcommand{\zzzz}{\zeta}
\newcommand{\eeee}{\epsilon}

\begin{document}
\nocite{*}
\title{{\bf \normalsize THREE-POINT COMPACT APPROXIMATION FOR THE CAPUTO FRACTIONAL DERIVATIVE }}
\author{Yuri Dimitrov\\
Department of Applied Mathematics and Statistics \\
University of Rousse, Rousse  7017, Bulgaria\\
\texttt{ymdimitrov@uni-ruse.bg}}
\maketitle
\begin{abstract} 
In this paper we derive the fourth-order asymptotic expansions of the trapezoidal approximation for the fractional integral
 and  the $L1$ approximation for the  Caputo derivative. We use the expansion of the $L1$ approximation to obtain the
three point compact approximation for the Caputo derivative
\begin{equation*}
\dddd{1}{\GGGG(2-\aaaa)h^\aaaa}\sum_{k=0}^{n} \dddddd_k^{(\aaaa)} y_{n-k}=\dddd{13}{12}y^{(\aaaa)}_n-\dddd{1}{6}y^{(\aaaa)}_{n-1}+\dddd{1}{12}y^{(\aaaa)}_{n-2}+O\llll(h^{3-\aaaa}\rrrr),
\end{equation*}
with weights $\dddddd_0^{(\aaaa)}=1-\zzzz(\aaaa-1),\; \dddddd_n^{(\aaaa)}=(n-1)^{1 -\aaaa}-n^{1-\aaaa},$
$$
\dddddd_1^{(\aaaa)}=2^{1-\aaaa}-2+2\zzzz(\aaaa-1),\; \dddddd_2^{(\aaaa)}=1-2^{2-\aaaa}+3^{1-\aaaa}-\zzzz(\aaaa-1),$$
$$\dddddd_k^{(\aaaa)}=(k-1)^{1-\aaaa}-2k^{1-a}+(k+1)^{1-\aaaa},\quad (k=3\cdots,n-1),$$
where  $y$ is a differentiable function which satisfies $y'(0)=0$. The numerical solutions of the fractional relaxation and the time-fractional  subdiffusion equations are discussed.

\noindent
{\bf 2010 Math Subject Classification:} 26A33, 34A08, 34E05, 41A25
\end{abstract}
\section{Introduction}\label{Intro} 
Fractional calculus is a rapidly growing field of science. 
Differential equations with the Caputo and Riemann-Liouville fractional derivatives  are used for modeling  processes in economics, physics and engineering \cite{Cartea2007,KaczorekRogowski2015,LimaFordLumb2014,Mainardi1996,MuslihAgrawalBaleanu2010}. 
The Caputo fractional derivative and the fractional integral of arbitrary order  are defined as a convolution of  the  derivatives of the function and the power function.
The fractional integral of order $\aaaa>0$ on the interval $[0,x]$ and the Caputo derivative  of order $n+\aaaa$, where $0<\aaaa<1$ and $n$ is a positive integer are defined as
$$I^\aaaa y(x)=\dddd{1}{\GGGG(\aaaa)}\int_0^x (x-\xi)^{\aaaa-1} y(\xi)d\xi,
$$
$$y^{(n+\aaaa)}(x)=D^{n+\aaaa} y(x)=\dddd{1}{\Gamma (1-\aaaa)}\int_0^x \dfrac{y^{(n+1)}(t)}{(x-t)^{\aaaa}}dt.$$
The Caputo derivative is expressed as a composition of the fractional integral of order $1-\aaaa$ and the derivative of the function of order $n+1$
$$y^{(n+\aaaa)}(x)=I^{1-\aaaa} y^{(n+1)}(x).$$
While analytical methods such as separation of variables, Laplace and Fourier transforms are used for solution of fractional differential equations, most of the equations used in practical applications are solved by a numerical method. The finite difference schemes for numerical solution of partial fractional differential equations use a discretization of the fractional derivative. 

Let $N$ be a positive integer and $\mathcal{X}_N=\{x_n=nh ; n=0,1,\cdots,N\}$ be a uniform mesh on the interval $[0,x]$ with step size $h=x/N$. Denote by $y_n$ the value of the function $y$ at the point $x_n$.
An  important approximation for the Caputo derivative is the  $L1$ approximation
\begin{equation}\label{L1}
y^{(\alpha)}_n  \approx \dfrac{1}{\GGGG(2-\aaaa)h^\alpha}\sum_{k=0}^{n} \ssss_k^{(\alpha)} y_{n-k},
\end{equation}
with weights $\ssss_0^{(\alpha)}=1$,  $\ssss_n^{(\alpha)}=(n-1)^{1-a}-n^{1-a}$,
$$\ssss_k^{(\alpha)}=(k+1)^{1-\alpha}-2k^{1-\alpha}+(k-1)^{1-\alpha}, \quad (k=2,\cdots,n-1).$$
 When the function $y$ has a continuous second derivative, the accuracy of the $L1$ approximation is $O(h^{2-\alpha})$ (\cite{LinXu2007}). The $L1$ approximation has been often used for numerical solution of ordinary and two-dimensional partial fractional differential equations.
The accuracy of the $L1$ approximation is insufficient for numerical solution of multidimensional  fractional differential equations which require a large number of computations, especially when the order of fractional differentiation is close to one. 
There is an increasing interest in high-order approximations for the Caputo and Riemann-Liouville fractional derivatives, which stems from the practical and the mathematical value of the fractional derivatives. 
In previous work \cite{Dimitrov2015_1} we determined the second-order expansion of the $L1$ approximation \eqref{2ndL1} and the second-order approximation for the Caputo derivative \eqref{ML1}, by modifying the first three coefficients of the $L1$ approximation with the value of the Riemann zeta function  $\zzzz(\aaaa-1)$,
\begin{equation}\label{ML1}
y^{(\aaaa)}_n=\dddd{1}{\GGGG(2-\aaaa)h^\aaaa}\sum_{k=0}^n \delta_k^{(\aaaa)} y_{n-k}+O\llll(h^{2}\rrrr),
\end{equation}
where $\dddddd_k^{(\aaaa)}=\ssss_k^{(\aaaa)}$ for $2\leq k\leq n$ and
$$\dddddd_0^{(\aaaa)}=\ssss_0^{(\aaaa)}-\zzzz(\aaaa-1),\; \dddddd_1^{(\aaaa)}=\ssss_1^{(\aaaa)}+2\zzzz(\aaaa-1),\; \dddddd_2^{(\aaaa)}=\ssss_2^{(\aaaa)}-\zzzz(\aaaa-1).$$
One approach for constructing high-order approximations for the Caputo and Riemann-Liouville fractional derivatives is based on  approximating the function $y$ on the subintervals of the mesh $\mathcal{X}_N$ by  Lagrange polynomials or spline interpolation \cite{Blank1996,CaoLiChen2015,Diethelm1997,LiChenYe2011,LiZhaoDengWu2014}. Another commonly used approach for approximation of the fractional derivative  is based on  the   expansion of the generating function and the Fourier transform of the approximation \cite{ChenDeng2014,ChenDeng2014_2,Dimitrov2014,LiDing2014,WuDingLi2014,ZhaoDeng2015}. The two approaches  are effective for constructing approximations for the fractional derivatives, although the formulas for the weights of the high-order approximations  may involve complex expressions. The compact approximations use an approximation for the fractional derivative and the higher accuracy is achieved by approximating a linear sum of fractional derivatives on adjacent nodes of the mesh. The compact approximations combine simpler expressions for the weights with a high-order accuracy of the approximation. Compact finite difference schemes for numerical solutions of partial fractional differential equations based on the Gr\"unwald and shifted  Gr\"unwald approximations are discussed in \cite{Dimitrov2014,JiSun2015,WangVong2014,YeLiuAnh2015,ZhaoDeng2015}. Yan,Pal and Ford \cite{YanPalFord2014} construct an approximation of the Caputo derivative and a numerical solution of the fractional relaxation equation with accuracy $O\llll(h^{3-\aaaa}\rrrr)$.
Gao, Sun and Zhang \cite{GaoSunZhang2014} use a quadratic interpolation on a three-point stencil to derive an $L1-2$ formula  approximation for the Caputo derivative  with accuracy $O\llll(h^{3-\aaaa}\rrrr)$. In recent work,  Li, Wu and Ding \cite{LiWuDing2015} derive an approximation for the Caputo derivative of order $3-\aaaa$ and compute the numerical solutions of the advection-diffusion equation with accuracy $O\llll(\tttt^{3-\aaaa}+h^2\rrrr)$
 and $O\llll(\tttt^{3-\aaaa}+h^4\rrrr)$. In \cite{Dimitrov2015_2} we discuss a method for improving the accuracy of the numerical solutions of the ordinary fractional relaxation equation 
\begin{equation*} 
y^{(\aaaa)}(x)+\llllll y(x)=F(x),
\end{equation*}
and the time-fractional subdiffusion equation
		\begin{equation*} 
\dfrac{\partial^\alpha u(x,t)}{\partial t^\alpha}=\dfrac{\partial^2 u(x,t)}{\partial x^2}+F(x,t),
	\end{equation*}
	when the solutions are nonsmooth functions at the initial point and we compare the numerical solutions  which use the $L1$ approximation \eqref{L1} and the modified $L1$ approximation \eqref{ML1} for the Caputo derivative.
The method is based on computing the fractional Taylor polynomials of the solution at the initial point of fractional differentiation.
In the present paper we extend the results of \cite{Dimitrov2015_1}. In Lemma 11 we determine the fourth-order expansion of the $L1$ approximation for the Caputo derivative
\begin{align*}
\dddd{1}{\GGGG(2-\aaaa)h^\aaaa}\sum_{k=0}^N \ssss_k^{(\aaaa)} y(x-kh)=y^{(\aaaa)}(x)+&\dddd{\zzzz(\aaaa-1)}{\GGGG(2-\aaaa)}y''(x)h^{2-\aaaa}\\
+\llll(\dddd{d^2}{dx^2}y^{(\aaaa)}(x)-\dddd{y'(0)}{\GGGG(-\aaaa)x^{1+\aaaa}}\rrrr)&\dddd{h^2}{12}-\dddd{\zzzz(\aaaa-2)}{\GGGG(2-\aaaa)}y'''(x)h^{3-\aaaa}\\
+\llll(\zzzz(\aaaa-3)+\dddd{\zzzz(\aaaa-1)}{6}\rrrr)&\dddd{y^{(4)}(x)}{2\GGGG(2-\aaaa)}h^{4-\aaaa}+O\llll(h^4\rrrr),
\end{align*}
and the three-point compact approximation
\begin{equation}\label{CML1}
\dddd{1}{\GGGG(2-\aaaa)h^\aaaa}\sum_{k=0}^{n} \dddddd_k^{(\aaaa)} y_{n-k}\approx \dddd{13}{12}y^{(\aaaa)}_n-\dddd{1}{6}y^{(\aaaa)}_{n-1}+\dddd{1}{12}y^{(\aaaa)}_{n-2}.
\end{equation}
When the function  $y$ is a differentiable function and $y'(0)=0$ the accuracy of compact approximation \eqref{CML1} is $O\llll(h^{3-\aaaa}\rrrr)$. 

The outline of the paper is as follows. In section 2 we introduce  the facts from Fractional calculus used in the rest of the paper. In section 3 we determine the fourth-order expansion \eqref{TE5th} of the trapezoidal  approximation for the fractional integral of order $2-\aaaa$. In section 4 we determine the fourth-order  expansion \eqref{4thL1} of the $L1$ approximation  for the Caputo derivative and the three-point compact approximation \eqref{CML1}. In section 5 we compare the numerical  solutions of the fractional relaxation and  subdiffusion equations which use approximations \eqref{L1}, \eqref{ML1} and \eqref{CML1}.
\section{Preliminaries}
The Caputo derivative of  order $\aaaa$, when $0<\aaaa<1$ is defined as
$$D^\aaaa y(x)=y^{(\aaaa)}(x)=\dddd{d^{\aaaa}y(x)}{d x^\aaaa}=\dddd{1}{\Gamma (1-\aaaa)}\int_0^x \dfrac{y'(t)}{(x-t)^{\aaaa}}dt.$$
Denote by $J^\aaaa y(x)$ the fractional integral
$$J^\aaaa y(x)=\int_0^x (x-t)^\aaaa y(t)dt.$$
When $\aaaa=0$ the fractional integral $J^0y(x)$ is equal to the definite integral of the function $y$ on the interval $[0,x]$.
The  fractional integrals $J^\aaaa y(x)$ and $I^\aaaa y(x)$ are related with
$$ J^{\aaaa} y(x)=\GGGG(1+\aaaa) I^{1+\aaaa} y(x).$$
In \cite{Dimitrov2015_1} we showed that 
\begin{equation*}
\GGGG(2-\aaaa) y^{(\aaaa)}(x)=\GGGG(2-\aaaa)I^{2-\aaaa}y^{\prime\prime}(x)+y^\prime (0)x^{1-\aaaa}.
\end{equation*}
Then
\begin{equation}\label{CDF}
\GGGG(2-\aaaa) y^{(\aaaa)}(x)=J^{1-\aaaa}y''(x)+x^{1-\aaaa}y'(0).
\end{equation}
The Miller-Ross sequential  derivative  is defined as  
$$y^{[\aaaa_1]}(x)=y^{(\aaaa_1)}(x),\quad y^{[\aaaa_1+\aaaa_2]}(x)=D^{\aaaa_1}D^{\aaaa_2}y(x),$$
$$y^{[\aaaa_1+\aaaa_2+\cdots+\aaaa_n]}(x)=D^{\aaaa_1}D^{\aaaa_2}\cdots D^{\aaaa_n}y(x).$$
Denote
 $$y^{[n\aaaa]}(x)=y^{[\aaaa+\aaaa+\cdots+\aaaa]}(x)=D^{\aaaa}D^{\aaaa}\cdots D^{\aaaa}y(x).$$
The fractional Taylor polynomials of degree $m$ for the  Miller-Ross derivatives $y^{[n\aaaa]}(0)$   are defined as
$$T^{(\aaaa)}_m(x)= \sum_{n=0}^m \dddd{y^{[n\aaaa]}(0)}{\GGGG(\aaaa n+1)}x^{\aaaa n}.$$
The fractional Taylor polynomials approximate the value of the function  
$y(h)\approx T^{(\aaaa)}_m (h)$, where $h$ is a small positive number.
The one-parameter and two-parameter Mittag-Leffler functions are defined for $\aaaa>0$ as
$$E_\alpha (x)=\sum_{n=0}^\infty \dfrac{x^n}{\Gamma(\alpha n+1)}, \quad
E_{\alpha,\beta} (x)=\sum_{n=0}^\infty \dfrac{x^n}{\Gamma(\alpha n+\beta)}.$$
When $\aaaa=\bbbb=1$ the Mittag-Leffler functions are equal to the exponential function.
The Caputo derivative and the fractional integral of the exponential function are expressed with the Mittag-Leffler functions as \cite{Ishteva2005}
$$D^\aaaa e^{\llllll x}=\llllll x^{1-\aaaa}E_{1,2-\aaaa}(\llllll x),\quad I^\aaaa e^{\llllll x}=x^{\aaaa}E_{1,1+\aaaa}(\llllll x).
$$
From the formula for the sine function
$$\sin x=\dddd{e^{i x}-e^{-i x}}{2i},$$
we obtain
$$D^\aaaa \sin (\llllll x)=\dddd{\llllll}{2} x^{1-\aaaa}\llll(E_{1,2-\aaaa}(i \llllll x)+E_{1,2-\aaaa}(-i \llllll x)\rrrr),$$
$$ I^\aaaa \sin (\llllll x)=-\dddd{i}{2}x^{\aaaa}\llll(E_{1,1+\aaaa}(i\llllll x)-E_{1,1+\aaaa}(-i\llllll x)\rrrr),
$$
and
$$J^{\aaaa}e^{\llllll x}=\GGGG (1+\aaaa) x^{1+\aaaa} E_{1,2+\aaaa}(\llllll x),$$
$$J^{\aaaa}\sin (\llllll x)=-\dddd{i}{2}\GGGG (1+\aaaa) x^{1+\aaaa}\llll(E_{1,2+\aaaa}(i\llllll x)-E_{1,2+\aaaa}(-i\llllll x)\rrrr).$$
The expressions for the analytical solutions of fractional differential equations  often contain the  Mittag-Leffler functions. The initial value problem,
	\begin{equation*} 
	\left\{
	\begin{array}{l l}
	y^{(\aaaa)}+\lambda y=0,\quad t>0,\quad  (n-1<\aaaa<n),&  \\
  y^{(k)}(0)=b_k,\qquad (k=0,1,\cdots,n-1),&  \\
	\end{array} 
		\right . 
	\end{equation*}
	has the solution 
	$$y(x)=\sum_{k=0}^{n-1} b_k x^k E_{\aaaa,k+1}\llll(-\lambda x^\aaaa\rrrr).$$
	The Gr\"unwald-Letnikov  fractional derivative is closely related to  Caputo and  Riemann-Liouville derivatives
$$\widetilde{D}^\alpha y(x) = \lim_{h\downarrow 0} \dfrac{1}{h^\alpha} \sum_{n=0}^{N} (-1)^n\binom{\alpha}{n} y(x-n h),$$
where
$$\binom{\alpha}{n}=\dfrac{\Gamma(\alpha+1)}{\Gamma(n+1)\Gamma(\alpha-n+1)}=\dfrac{\alpha(\alpha-1)\cdots(\alpha-n+1)}{n!}.$$
Denote $\oooo_n^{(\alpha)}=(-1)^n\binom{\alpha}{n}$. The Gr\"unwald formula approximation for the Caputo derivative  has  a first-order accuracy
\begin{equation*} 
\dddd{1}{h^\aaaa}\sum_{k=0}^{N} \oooo_n^{(\alpha)} y(x-k h)=y^{(\alpha )}(x)+O\left(h\right).
\end{equation*}
When the function $y$ is differentiable  and $y(0)=y'(0)=y''(0)=0$, the Gr\"unwald formula approximates the Caputo derivative at the point $x-\alpha  h/2$ with  second-order accuracy
\begin{equation*} 
\dddd{1}{h^\aaaa}\sum_{k=0}^{N} \oooo_n^{(\alpha)} y(x-k h)=y^{(\alpha )}\left(x-\dfrac{\alpha  h}{2}\right)+O\left(h^2\right).
\end{equation*}
The Gr\"unwald formula approximation for the Caputo derivative has a fourth-order expansion
\begin{align*} 
\dddd{1}{h^\aaaa}\sum_{k=0}^{N} \oooo_n^{(\alpha)} y(x-k h)=y^{(\alpha)}(x)-\dddd{\aaaa}{2}y&^{[1+\alpha]}(x) h+\dddd{\aaaa(1+3\aaaa)}{24}y^{[2+\alpha]}(x) h^2\\
&+\dddd{\aaaa^2(1+\aaaa)}{48}y^{[3+\alpha]}(x) h^3+O\left(h^4\right),
\end{align*}
when $y(0)=y'(0)=y''(0)=y'''(0)=0$. In \cite{Dimitrov2014} we discuss the two-point and three-point compact Gr\"unwald approximations for the Caputo derivative and the numerical solutions of the fractional relaxation and subdiffusion equations.
\begin{equation*} 
\dddd{1}{h^\aaaa}\sum_{k=0}^{n} \oooo_k^{(\alpha)} y_{n-k}=\llll(\dfrac{\aaaa}{2}  \rrrr)y^{(\alpha )}_{n-1}+
\llll(1-\dfrac{\aaaa}{2} \rrrr)y^{(\alpha )}_{n}+O\left(h^2\right),
\end{equation*}
\begin{align}  \label{3PTG}
\dddd{1}{h^\aaaa}\sum_{k=0}^{n} \oooo_k^{(\alpha)} y_{n-k}=\llll(\dddd{a^2}{8}-\dddd{5a}{24}\rrrr)&y^{(\aaaa)}_{n-2}+\llll(\dddd{11a}{12}-\dddd{a^2}{4}\rrrr)y^{(\aaaa)}_{n-1}\\
&+\llll(1-\dddd{17a}{24}+\dddd{a^2}{8}\rrrr)y^{(\aaaa)}_{n}+O\llll( h^3\rrrr).\nonumber
\end{align}
In the  present paper we derive the fourth-order expansion \eqref{TE5th} of the trapezoidal approximation for the fractional integral $J^\aaaa y(x)$ and three-point compact approximation \eqref{CML1} for the modified $L1$ approximation. The trapezoidal sum of the function $y$ on the interval $[0,x]$ is defined as
$$\mathcal{T}_h[y]=\dddd{h}{2}\llll(y(0)+\sum_{k=1}^{n-1}y(kh)+y(x)   \rrrr).$$
When the function $y$ is a differentiable function the trapezoidal sum is a second-order approximation for the definite integral on the interval $[0,x]$,
$$\mathcal{T}_h[y]=\int_0^x y(t)dt+O\llll(h^2 \rrrr).
$$
The asymptotic expansion of the trapezoidal approximation for the definite integral is given by the Euler-MacLaurin formula  \cite{AbramowitzStegun1964,Kouba2013,Lampret2001},
$$\mathcal{T}_h[y]=\int_0^x y(\xi)d\xi+\sum_{n=1}^\infty \dddd{B_{2n}}{(2n)!}\llll(y^{(2n-1)}(x) -y^{(2n-1)}(0)\rrrr)h^{2n}.
$$
The Bernouli numbers $B_n$ with an odd index are equal to zero, and the Bernoulli numbers with an even index are expressed with the values of the Riemann zeta function at the even integrs
$$B_{2n}=(-1)^{n+1}\dddd{2(2n)!}{(2\pi)^{2n}}\zzzz(2n).$$
The Bernoulli numbers have a high growth rate and the above sum diverges for many standard calculus functions. The properties of the Bernoulli numbers and the Bernoulli polynomials are discussed  in Kouba \cite{Kouba2013}.
The asymptotic expansion formula for the sum of the powers of the first $n$ integers is related to the Euler-MacLaurin formula \cite{AbramowitzStegun1964}
\begin{equation}\label{Expansion1}
\sum_{k=1}^{n-1}k^\aaaa=\zzzz(-\aaaa)+\dddd{n^{1+\aaaa}}{1+\aaaa}
\sum_{m=0}^{\infty}\binom{\aaaa+1}{m}\dddd{B_m}{n^m}.
\end{equation}
In \cite{Dimitrov2015_1} we use expansion formula  \eqref{Expansion1} to derive the second order expansions
\begin{equation*}
h^{1+\aaaa}\sum_{k=0}^N k^{\aaaa} y(x-kh)=J^{\aaaa}y(x)+\dddd{y(0)x^\aaaa}{2}h+\zzzz(-\aaaa) y(x)h^{1+\aaaa}+O\llll(h^2\rrrr),
\end{equation*}
\begin{equation}\label{2ndL1}
\dddd{1}{\GGGG(2-\aaaa)h^\aaaa}\sum_{k=0}^N \ssss_k^{(\aaaa)} y(x-kh)=y^{(\aaaa)}(x)+\dddd{\zzzz(\aaaa-1)}{\GGGG(2-\aaaa)}y^{\prime\prime}(x) h^{2-\aaaa}+O\llll(h^2\rrrr).
\end{equation}
In the present paper we extend the results from \cite{Dimitrov2015_1}. We determine the fourth-order expansions
\begin{align}\label{TE5th}
h^{1+\aaaa}\sum_{k=0}^N k^{\aaaa} y(x-kh)-\dddd{y(0)x^\aaaa}{2}h=J^{\aaaa}&y(x)+\zzzz(-\aaaa) y(x)h^{1+\aaaa}\\
+\dddd{\aaaa x^{\aaaa-1}y(0)-x^\aaaa y'(0)}{12}&h^2-y'(x)\zzzz(-1-\aaaa)h^{2+\aaaa}\nonumber\\
+\dddd{y''(x)}{2}\zzzz(-2-\aaaa)&h^{3+\aaaa}+O\llll(h^4 \rrrr),\nonumber
\end{align}
\begin{align}\label{4thL1}
\dddd{1}{\GGGG(2-\aaaa)h^\aaaa}\sum_{k=0}^N \ssss_k^{(\aaaa)}& y(x-kh)=y^{(\aaaa)}(x)+y^{[2+\aaaa]}(x)\dddd{h^2}{12}\\
+\dddd{\zzzz(\aaaa-1)}{\GGGG(2-\aaaa)}&y''(x)h^{2-\aaaa}-\dddd{y'(0)}{\GGGG(-\aaaa)x^{1+\aaaa}}\dddd{h^2}{12}-\dddd{\zzzz(\aaaa-2)}{\GGGG(2-\aaaa)}y'''(x)h^{3-\aaaa}\nonumber\\
&+\llll(\zzzz(\aaaa-3)+\dddd{\zzzz(\aaaa-1)}{6}\rrrr)\dddd{y^{(4)}(x)}{2\GGGG(2-\aaaa)}h^{4-\aaaa}+O\llll(h^4\rrrr),\nonumber
\end{align}
and the three-point compact approximation for the Caputo derivative
\begin{equation*}
\dddd{1}{\GGGG(2-\aaaa)h^\aaaa}\sum_{k=0}^{n} \dddddd_k^{(\aaaa)} y_{n-k}\approx \dddd{1}{12}y^{(\aaaa)}_{n-2}-\dddd{1}{6}y^{(\aaaa)}_{n-1}+\dddd{13}{12}y^{(\aaaa)}_n.
\end{equation*}
The   accuracy of the approximation is  $O\llll( h^{3-\aaaa}\rrrr)$ when  $y'(0)=0$.
The functions  $y(x)=e^x-x$  and $y(x)=\cos x$ are differentiable functions and satisfy the condition $y'(0)=0$. In Table 1 and Table 2 we compute the error and  the order of compact approximation \eqref{CML1} and  fourth-order expansion \eqref{TE5th}.
\begin{table}[ht]
    \caption{Error and order of approximation \eqref{CML1}  for the functions  $y(t)=e^t-t$ (left) with $\aaaa=0.5,x=2$, and $y(t)=\cos t,\aaaa=0.75,x=1$ (right).}
    \begin{subtable}{0.5\linewidth}
      \centering
  \begin{tabular}{l c c }
  \hline \hline
    $\boldsymbol{h}$ & $\mathbf{Error}$ & $\mathbf{Order}$  \\ 
		\hline \hline
$0.05$         &$0.000751014$            &$2.397394$\\
$0.025$        &$0.000138643$            &$2.437469$\\
$0.0125$       &$0.000025183$            &$2.460848$\\
$0.00625$      &$4.53\times 10^{-6}$     &$2.474854$\\
$0.003125$     &$8.10\times 10^{-7}$     &$2.483489$\\
		\hline
  \end{tabular}
    \end{subtable}%
    \begin{subtable}{.5\linewidth}
      \centering
				\quad
  \begin{tabular}{ l  c  c }
    \hline \hline
    $\boldsymbol{h}$ & $\mathbf{Error}$ &$\mathbf{Order}$  \\ \hline \hline
$0.05$     & $0.0002819560$         & $2.178218$ \\
$0.025$    & $0.0000607879$         & $2.213615$ \\
$0.0125$   & $0.0000129516$         & $2.230656$ \\
$0.00625$  & $2.74\times 10^{-6}$   & $2.239427$ \\
$0.003125$ & $5.79\times 10^{-7}$   & $2.244120$ \\   
\hline
  \end{tabular}
    \end{subtable} 
\end{table}
\begin{table}[ht]
    \caption{Error and order of approximation \eqref{TE5th}  for the functions  $y(t)=e^t$ and $\aaaa=0.75,x=2$ (left) and $y(t)=\sin t,\aaaa=0.25,x=1$ (right).}
    \begin{subtable}{0.5\linewidth}
      \centering
  \begin{tabular}{l c c }
  \hline \hline
    $\boldsymbol{h}$ & $\mathbf{Error}$ & $\mathbf{Order}$  \\ 
		\hline \hline
$0.05$         &$1.18\times 10^{-7}$      &$4.389900$\\
$0.025$        &$6.19\times 10^{-9}$      &$4.274558$\\
$0.0125$       &$3.35\times 10^{-10}$     &$4.185221$\\
$0.00625$      &$1.92\times 10^{-11}$     &$4.120378$\\
$0.003125$     &$1.14\times 10^{-12}$     &$4.076583$\\
		\hline
  \end{tabular}
    \end{subtable}%
    \begin{subtable}{.5\linewidth}
      \centering
				\quad
  \begin{tabular}{ l  c  c }
    \hline \hline
    $\boldsymbol{h}$ & $\mathbf{Error}$ &$\mathbf{Order}$  \\ \hline \hline
$0.05$     & $1.18\times 10^{-8}$    & $3.959934$ \\
$0.025$    & $7.55\times 10^{-10}$   & $3.966716$ \\
$0.0125$   & $4.81\times 10^{-11}$   & $3.972384$ \\
$0.00625$  & $3.06\times 10^{-12}$   & $3.977080$ \\
$0.003125$ & $1.94\times 10^{-13}$   & $3.980632$ \\       
\hline
  \end{tabular}
    \end{subtable} 
\end{table}
\section{Fourth-Order Expansion of the Trapezoidal Approximation  for the Fractional Integral}
In the present section we determine the fourth-order expansion \eqref{TE5th} of the trapezoidal approximation for the fractional integral
$$J^\aaaa y(x)=\int_0^x (x-t)^\aaaa y(t) dt,$$
where $0<\aaaa<1$. 
When the function $y$ is a differentiable function on the interval $[0,x]$, the trapezoidal approximation for the definite integral has second-order accuracy $O\llll(h^{2}\rrrr)$. From the Euler-MacLaurin formula the trapezoidal sum of the function $y$ has a  fourth-order asymptotic expansion
$$\mathcal{T}_h[y]=\int_0^x y(t)dt+\dddd{h^2}{12} \llll(y'(x)-y'(0)\rrrr)+O\llll(h^4 \rrrr).$$
The function $z(t)=(x-t)^\aaaa y(t)$ has a  singularity on  $[0,x]$, because its first derivative is unbounded at the point $t=x$. In Theorem 5 we determine the fourth-order expansion  of the trapezoidal approximation for the fractional integral $J^\aaaa y(x)$. The expansion formula contains the term $\zzzz(-\aaaa) y(x)h^{1+\aaaa}$. When $y(x)\neq 0$ the trapezoidal approximation for the definite integral of the function $z$  has accuracy $O\llll(h^{1+\aaaa}\rrrr)$ and   the trapezoidal approximation has second-order accuracy  if $y(x)=0$.  In Lemma 1, Lemma 2 and Lemma 3 we determine the fourth-order expansions of the trapezoidal approximations for the functions 
$$(x-t)^\aaaa,\quad (x-t)^{1+\aaaa},\quad (x-t)^{2+\aaaa},$$
 on the interval $[0,x]$. The proofs of the lemmas use the expansion formula approximation for the sum of the powers of the first $n$ integers
\begin{equation}\label{SPE2}
\sum_{k=1}^n k^\aaaa=\zzzz(-\aaaa)+\dddd{n^{1+\aaaa}}{1+\aaaa}+\dddd{n^\aaaa}{2}+\dddd{\aaaa}{12}n^{\aaaa-1}+O\llll(\dddd{1}{n^{3-\aaaa}}  \rrrr).
\end{equation}
\begin{lem} Let $y(t)=1$ and $z(t)=(x-t)^\aaaa$. Then
$$
\mathcal{T}_h[z]=J^\aaaa y(x)+\zzzz(-\aaaa)h^{1+\aaaa}+\dddd{\aaaa x^{\aaaa-1}}{12}h^2+O\llll(h^4 \rrrr).
$$
\end{lem}
\begin{proof} The function $z$ has trapezoidal sum
$$\mathcal{T}_h[z]=h^{1+\aaaa}\sum_{k=1}^n k^\aaaa-\dddd{1}{2}x^\aaaa h,$$
and $J^\aaaa y(x)=\frac{x^{1+\aaaa}}{1+\aaaa}$. From  \eqref{SPE2}
$$\mathcal{T}_h[z]=\dddd{x^{1+\aaaa}}{1+\aaaa}+\dddd{1}{2}x^\aaaa h+\zzzz (-\aaaa) h^{1+\aaaa}+\dddd{\aaaa}{12}x^{\aaaa-1} h^2-\dddd{1}{2}x^\aaaa h+O\llll(\dddd{h^{1+\aaaa}}{n^{3-\aaaa}}  \rrrr).$$
We have that $h^{1+\aaaa}/n^{3-\aaaa}=h^{4}/(nh)^{3-\aaaa}=h^{4}/x^{3-\aaaa}$. Then
$$\mathcal{T}_h[z]=J^\aaaa y(x)+\zzzz (-\aaaa) h^{1+\aaaa}+\dddd{\aaaa}{12}x^{\aaaa-1} h^2+O\llll(h^4\rrrr).$$
\end{proof}
\begin{lem} Let $y(t)=(x-t)^\aaaa$ and $z(t)=(x-t)^{1+\aaaa}$. Then
$$\mathcal{T}_h[z]=J^\aaaa y(x)+\dddd{(1+\aaaa) x^{\aaaa}}{12}h^{2}+\zzzz(-1-\aaaa)h^{2+\aaaa}+O\llll(h^4 \rrrr).$$
\end{lem}
\begin{proof}  From expansion \eqref{SPE2} 
$$\sum_{k=1}^n k^{1+\aaaa}=\zzzz(-1-\aaaa)+\dddd{n^{2+\aaaa}}{2+\aaaa}+\dddd{n^{1+\aaaa}}{2}+\dddd{1+\aaaa}{12}n^{\aaaa}+O\llll(\dddd{1}{n^{2-\aaaa}}  \rrrr).$$
We have that $J^\aaaa y(x)=\frac{x^{2+\aaaa}}{2+\aaaa}$ and
$$\mathcal{T}_h[z]=h^{2+\aaaa}\sum_{k=1}^n k^{1+\aaaa}-\dddd{x^{1+\aaaa}}{2} h.$$
Then
$$\mathcal{T}_n[z]=\dddd{x^{2+\aaaa}}{2+\aaaa}+\dddd{x^{1+\aaaa}}{2} h+\dddd{1+\aaaa}{12}x^{\aaaa} h^2+\zzzz (-1-\aaaa) h^{2+\aaaa}-\dddd{x^{1+\aaaa}}{2} h+O\llll(\dddd{h^{2+\aaaa}}{n^{2-\aaaa}}  \rrrr),$$
$$\mathcal{T}_h[z]=J^\aaaa y(x)+\dddd{1+\aaaa}{12}x^{\aaaa} h^2+\zzzz (-1-\aaaa) h^{2+\aaaa}+O\llll(h^4\rrrr),$$
because $h^{2+\aaaa}/n^{2-\aaaa}=h^{4}/(nh)^{2-\aaaa}=h^{4}/x^{2-\aaaa}$. 
\end{proof}
\begin{lem} Let $y(t)=(x-t)^{1+\aaaa}$ and $z(t)=(x-t)^{2+\aaaa}$. Then
$$T_h[z]=J^\aaaa y(x)+\dddd{(2+\aaaa) x^{1+\aaaa}}{12}h^{2}+\zzzz(-2-\aaaa)h^{3+\aaaa}+O\llll(h^4 \rrrr).$$
\end{lem}
\begin{proof}  From \eqref{SPE2} 
$$\sum_{k=1}^n k^{2+\aaaa}=\zzzz(-2-\aaaa)+\dddd{n^{3+\aaaa}}{3+\aaaa}+\dddd{n^{2+\aaaa}}{2}+\dddd{2+\aaaa}{12}n^{1+\aaaa}+O\llll(\dddd{1}{n^{1-\aaaa}}  \rrrr).$$
We have that $J^\aaaa y(x)=\frac{x^{3+\aaaa}}{3+\aaaa}$ and
$$\mathcal{T}_h[z]=h^{3+\aaaa}\sum_{k=1}^n k^{2+\aaaa}-\dddd{x^{2+\aaaa}}{2} h.$$
Then
$$\mathcal{T}_h[z]=\dddd{x^{3+\aaaa}}{3+\aaaa}+\dddd{x^{2+\aaaa}}{2} h+\dddd{2+\aaaa}{12}x^{1+\aaaa} h^2+\zzzz (-2-\aaaa) h^{2+\aaaa}-\dddd{x^{2+\aaaa}}{2} h+O\llll(\dddd{h^{3+\aaaa}}{n^{1-\aaaa}}  \rrrr),$$
$$\mathcal{T}_n[z]=J^\aaaa y(x)+\dddd{2+\aaaa}{12}x^{1+\aaaa} h^2+\zzzz (-2-\aaaa) h^{3+\aaaa}+O\llll(h^4\rrrr),$$
because $h^{3+\aaaa}/n^{1-\aaaa}=h^{4}/(n h)^{1-\aaaa}=h^{4}/x^{1-\aaaa}$. 
\end{proof}
In the next theorem we determine the fourth-order expansion of the trapezoidal approximation for the fractional integral $J^\aaaa y(x)$, when the function $y(x)$ is a polynomial.
\begin{thm} Let $y(t)$ be a polynomial and $z(t)=(x-t)^\aaaa y(t)$. Then
\begin{align*}
\mathcal{T}_h[z]=J^\aaaa y(x)+y(x)\zzzz(-\aaaa)h^{1+\aaaa}-\dddd{z'(0)}{12}&h^2-y'(x)\zzzz(-1-\aaaa)h^{2+\aaaa}\\
&+\dddd{y''(x)}{2}\zzzz(-2-\aaaa)h^{3+\aaaa}+O\llll(h^4 \rrrr).\nonumber
\end{align*}
\end{thm}
\begin{proof} Let $y(t)$ be a polynomial of degree $m$. From Taylor expansion of the function $y(t)$ at the point $t=x$
$$y(t)=\sum_{k=0}^m p_k (x-t)^k=p_0+p_1 (x-t)+p_2 (x-t)^2+ \sum_{k=3}^m p_k (x-t)^k.$$
The coefficients $p_k$ are expressed with the values of the derivatives of the function $y(t)$ at the point $t=x$, 
\begin{equation}\label{coefs}
p_k=(-1)^k y^{(k)}(x)/k!.
\end{equation}
 Denote
$$y_1(t)=p_0+p_1 (x-t)+p_2 (x-t)^2,\quad y_2(t)= \sum_{k=3}^m p_k (x-t)^k,$$
$$z_1(t)=(x-t)^\aaaa y_1(t),\quad z_2(t)=(x-t)^\aaaa y_2(t).$$
Then
$$y(t)=y_1(t)+y_2(t),\quad z(t)=z_1(t)+z_2(t).$$
 The function $z(t)$ has a  first derivative  at the point $t=0$,
$$z(t)=\sum_{k=0}^m p_k (x-t)^{\aaaa+k}=p_0 (x-t)^{\aaaa}+p_1 (x-t)^{1+\aaaa}+p_2 (x-t)^{2+\aaaa}+ z_2(t),$$
$$z'(t)=-\aaaa p_0 (x-t)^{\aaaa-1}-(1+\aaaa)p_1 (x-t)^{\aaaa}-(2+\aaaa)p_2 (x-t)^{1+\aaaa}+ z_2'(t),$$
\begin{equation}\label{Z0}
z'(0)=-\aaaa p_0 x^{\aaaa-1}-(1+\aaaa)p_1 x^{\aaaa}-(2+\aaaa)p_2 x^{\aaaa+1}+ z_2'(0).
\end{equation}
The function $z_2(t)$ has a continuous and bounded third derivative on $[0,x]$ and $z_2'(x)=0$. From the Euler-MacLaurin formula, the trapezoidal sum of the function $z_2(t)$ has a fourth-order expansion
\begin{equation}\label{Z2}
\mathcal{T}_h[z_2]=J^\aaaa y_2(x)-\dddd{h^2}{12}z_2'(0)+O\llll( h^4 \rrrr).
\end{equation}
 Now we determine the fourth-order expansion for the trapezoidal sum of the function $z_1(t)$. From  Lemma 1, Lemma 2 and Lemma 3, 
$$\mathcal{T}_h[z_1]=p_0 \mathcal{T}_h[(x-t)^\aaaa]+p_1 \mathcal{T}_h[(x-t)^{1+\aaaa}]+p_2 \mathcal{T}_h[(x-t)^{2+\aaaa}],$$
\begin{align}\label{Z1}
\mathcal{T}_h[z_1]=p_0 &\llll( \dddd{x^{1+\aaaa}}{1+\aaaa}+\zzzz(-\aaaa)h^{1+\aaaa}+\dddd{\aaaa x^{\aaaa-1}}{12}h^2 \rrrr)\\
&+p_1 \llll( \dddd{x^{2+\aaaa}}{2+\aaaa}+\dddd{(1+\aaaa) x^{\aaaa}}{12}h^{2}+\zzzz(-1-\aaaa)h^{2+\aaaa} \rrrr)\nonumber\\
&+p_2\llll( \dddd{x^{3+\aaaa}}{3+\aaaa}+\dddd{(2+\aaaa) x^{1+\aaaa}}{12}h^{2}+\zzzz(-2-\aaaa)h^{3+\aaaa}\rrrr)+O\llll(h^4 \rrrr).\nonumber
\end{align}
From \eqref{Z0}, \eqref{Z2} and \eqref{Z1} we obtain
\begin{align*}
\mathcal{T}_h[z]=J^\aaaa y(x)+&p_0\zzzz(-\aaaa)h^{1+\aaaa}-\dddd{z'(0)}{12}h^2\\
&+p_1\zzzz(-1-\aaaa)h^{2+\aaaa}+p_2\zzzz(-2-\aaaa)h^{3+\aaaa}+O\llll(h^4 \rrrr).
\end{align*}
From the formula \eqref{coefs} for the coefficients of $y(t)$,
$$p_0=y(x),\quad p_1=-y'(x),\quad p_2=\dddd{y''(x)}{2}.$$
Then
\begin{align*}
\mathcal{T}_h[z]=J^\aaaa y(x)+&y(x)\zzzz(-\aaaa)h^{1+\aaaa}-\dddd{z'(0)}{12}h^2\\
&-y'(x)\zzzz(-1-\aaaa)h^{2+\aaaa}+\dddd{y''(x)}{2}\zzzz(-2-\aaaa)h^{3+\aaaa}+O\llll(h^4 \rrrr).
\end{align*}
\end{proof}
Now we extend the result from Theorem 4 to any sufficiently differentiable function $y$ on the interval $[0,x]$. By the Weierstrass Approximation Theorem every continuous function on the interval $[0,x]$ is uniformly approximated by polynomials:  for every  $\eeee>0$ there exists a polynomial $p(t)$ such that
$$|y(t)-p(t)|<\eeee.$$
We can show that if $y(t)$ is a sufficiently differentiable function, there  exists a polynomial $p(t)$ such that
\begin{equation} \label{pe}
|y(t)-p(t)|<\eeee,\quad |y'(t)-p'(t)|<\eeee,\quad |y''(t)-p''(t)|<\eeee,
\end{equation}
for all $t\in [0,x]$. The proof is similar to the proof of Theorem 8 in \cite{Dimitrov2015_1}.
\begin{thm} Let $y(t)$ be a sufficiently differentiable function. Then
\begin{align*}
h^{1+\aaaa}\sum_{k=0}^n k^{\aaaa}y(x-&kh)=J^{\aaaa}y(x)+\dddd{y(0)x^\aaaa}{2}h+\zzzz(-\aaaa) y(x)h^{1+\aaaa}\\
&+\dddd{\aaaa x^{\aaaa-1}y(0)-x^\aaaa y'(0)}{12}h^2-y'(x)\zzzz(-1-\aaaa)h^{2+\aaaa}\\
&+\dddd{y''(x)}{2}\zzzz(-2-\aaaa)h^{3+\aaaa}+O\llll(h^4 \rrrr).
\end{align*}
\end{thm}
\begin{proof} Let $z(t)=(x-t)^\aaaa y(t)$. Then
$$z'(t)=-\aaaa (x-t)^{\aaaa-1}y(t)+(x-t)^\aaaa y'(t),$$
$$z'(0)=-\aaaa x^{\aaaa-1}y(0)+x^\aaaa y'(0).$$
The function $z$ has a trapezoidal sum
$$\mathcal{T}_h[z]=h^{1+\aaaa}\sum_{k=0}^n k^\aaaa y(x-k h)-\dddd{x^\aaaa y(0)}{2}h.$$
Denote
$$\mathcal{Q}_h[y]=h^{1+\aaaa}\sum_{k=0}^n k^\aaaa y(x-k h),$$
\begin{align*}
\mathcal{R}_h[y]=J^{\aaaa}y(x)+&\dddd{y(0)x^\aaaa}{2}h+\zzzz(-\aaaa) y(x)h^{1+\aaaa}\\
&+\dddd{\aaaa x^{\aaaa-1}y(0)-x^\aaaa y'(0)}{12}h^2-y'(x)\zzzz(-1-\aaaa)h^{2+\aaaa}\\
&+\dddd{y''(x)}{2}\zzzz(-2-\aaaa)h^{3+\aaaa}.
\end{align*}
Let $\eeee>0$ and $p(t)$ be a polynomial which satisfies \eqref{pe}. From Theorem 4
\begin{equation}\label{F1}
\mathcal{Q}_h[p]-\mathcal{R}_h[p]=O\llll(h^4 \rrrr).
\end{equation}
From the expansion formula for the sum of powers \eqref{Expansion1},
\begin{equation}\label{SPApprox}
\sum_{k=0}^n k^\aaaa=\dddd{n^{1+\aaaa}}{1+\aaaa}+O\llll(n^\aaaa \rrrr) =O\llll(n^{1+\aaaa} \rrrr).
\end{equation}
$\bullet$ {\it Now we estimate} $\mathcal{Q}_h[y]-\mathcal{Q}_h[p]$.
$$ \llll|\mathcal{Q}_h[y]-\mathcal{Q}_h[p]\rrrr|=\llll| h^{1+\aaaa}\sum_{k=0}^n k^\aaaa (y(x-k h) -p(x-k h))  \rrrr|\leq \eeee h^{1+\aaaa}\sum_{k=0}^n k^\aaaa.$$
Then
\begin{equation*}
 |\mathcal{Q}_h[y]-\mathcal{Q}_h[p]|\leq \eeee h^{1+\aaaa}\sum_{k=0}^n k^\aaaa = \eeee h^{1+\aaaa}\llll( \dddd{n^{1+\aaaa}}{1+\aaaa}+O\llll(n^\aaaa \rrrr)  \rrrr),
\end{equation*}
$$ |\mathcal{Q}_h[y]-\mathcal{Q}_h[p]|\leq \eeee  \dddd{x^{1+\aaaa}}{1+\aaaa}+\eeee O\llll(x^\aaaa h\rrrr),  $$
\begin{equation}\label{F2}
 |\mathcal{Q}_h[y]-\mathcal{Q}_h[p]|= O\llll(\eeee\rrrr)  .
\end{equation}
$\bullet$ {\it Estimate for} $J^\aaaa y(x)-J^\aaaa p(x)$.
$$\llll|J^\aaaa y(x)-J^\aaaa p(x)\rrrr|=\llll|\int_0^x (x-t)^\aaaa (y(t)-p(t))dt\rrrr|\leq \int_0^x (x-t)^\aaaa |y(t)-p(t)|dt,$$
$$\llll|J^\aaaa y(x)-J^\aaaa p(x)\rrrr|\leq \eeee \int_0^x (x-t)^\aaaa dt=-\eeee\llll. \dddd{(x-t)^{1+\aaaa}}{1+\aaaa}\rrrr|_0^x,$$
$$\llll|J^\aaaa y(x)-J^\aaaa p(x)\rrrr|\leq   \llll(\dddd{x^{1+\aaaa}}{1+\aaaa}\rrrr) \eeee. $$
$\bullet$ {\it Estimate for} $\mathcal{R}_h[y]-\mathcal{R}_h[p].$
\begin{align*}
|\mathcal{R}_h[y]-\mathcal{R}_h &[p]|\leq \llll|J^\aaaa y(x)-J^\aaaa p(x)\rrrr|+\llll|y(0)-p(0)\rrrr| \dddd{x^\aaaa h}{2}\\
& + |y(x)-p(x)|\zzzz(-\aaaa)h^{1+\aaaa}+\llll|y'(x)-p'(x)\rrrr|\zzzz(-1-\aaaa)h^{2+\aaaa}\\
&+\llll(\aaaa x^{\aaaa-1}|y(0)-p(0)|+x^\aaaa |y'(0)-p'(0)|\rrrr)\dddd{h^2}{12}\\
&+|y''(x)-p''(x)|\dddd{\zzzz(-2-\aaaa)}{2}h^{3+\aaaa},
\end{align*}
\begin{align*}
|\mathcal{R}_h[y]-\mathcal{R}_h[p]| \leq \eeee \bigg( \dddd{x^{1+\aaaa}}{1+\aaaa}+& \dddd{x^\aaaa h}{2}+\zzzz(-\aaaa)h^{1+\aaaa}+\zzzz(-1-\aaaa)h^{2+\aaaa}\nonumber\\
&+\dddd{\zzzz(-2-\aaaa)h^{3+\aaaa}}{2}+\aaaa x^{\aaaa-1}\dddd{h^2}{12}+x^\aaaa\dddd{h^2}{12} \bigg),
\end{align*}
\begin{align}\label{F3}
|\mathcal{R}_h[y]-\mathcal{R}_h[p]|=O(\eeee).
\end{align}
$\bullet$ {\it Estimate for}  $\mathcal{R}_h[y]-\mathcal{Q}_h[y]$.
$$|\mathcal{R}_h[y]-\mathcal{Q}_h[y]|\leq |\mathcal{R}_h[y]-\mathcal{R}_h[p]|+|\mathcal{R}_h[p]-\mathcal{Q}_h[p]|+|\mathcal{Q}_h[p]-\mathcal{Q}_h[y]|.
$$
From \eqref{F1},\eqref{F2} and \eqref{F3} we obtain
$$|\mathcal{R}_h[y]-\mathcal{Q}_h[y]|\leq O(\eeee)+O\llll(h^4 \rrrr)+O(\eeee)=O\llll(\eeee+h^4 \rrrr).
$$
By letting $\eeee \rightarrow 0$ we obtain
$$|\mathcal{R}_h[y]-\mathcal{Q}_h[y]|=O\llll(h^4 \rrrr).
$$
\end{proof}
The Caputo derivative of the function $y$ is expressed \eqref{CDF} as the  composition of the fractional integral $J^{1-\aaaa}$ and the second derivative of the function $y$. The expansion formula \eqref{J1a} for the trapezoidal approximation of the fractional integral $J^{1-\aaaa}$ is a direct consequence of Theorem 5. 
In Lemma 10 we use \eqref{J1a}  to derive the fourth-order expansion of the $L1$ approximation for the Caputo derivative.
\begin{cor} Let $y(t)$ be a sufficiently differentiable function. Then
\begin{align}\label{J1a}
h^{2-\aaaa}\sum_{k=0}^{n-1} k^{1-\aaaa} &y(x-kh)=J^{1-\aaaa}y(x)-\dddd{y(0)x^{1-\aaaa}h}{2}+\zzzz(\aaaa-1) y(x)h^{2-\aaaa}\nonumber\\
&+\dddd{(1-\aaaa) x^{-\aaaa}y(0)-x^{1-\aaaa} y'(0)}{12}h^2-y'(x)\zzzz(\aaaa-2)h^{3-\aaaa}\nonumber\\
&+\dddd{y''(x)}{2}\zzzz(\aaaa-3)h^{4-\aaaa}+O\llll(h^4 \rrrr).
\end{align}
\end{cor}
\section{Expansion of the L1 Approximation for the Caputo Derivative}
In \cite{Dimitrov2015_1} we determined the second-order expansion \eqref{2ndL1} of the $L1$ approximation for the Caputo derivative. In this section we extend the method of \cite{Dimitrov2015_1} to determine the fourth-order expansion \eqref{4thL1} of the $L1$ approximation and the three-point compact approximation \eqref{CML1} for the Caputo derivative. 
Denote by $\mathcal{L}_{h}[y]$ the non-constant part of the $L1$ approximation
$$\mathcal{L}_{h}[y]=\dddd{1}{h^\aaaa} \sum_{k=0}^{N} \ssss_k^{(\alpha)} y_{N-k}. $$
Let $\Delta_h^1 y_n$ and $\Delta_h^2 y_n$ be the forward difference and the central difference of the function $y(x)$ at the point $x_n=n h,$
$$\Delta_h^1 y_n=y_{n+1}-y_n,\quad \Delta_h^2 y_n=y_{n+1}-2y_n+y_{n-1}.$$
In Lemma 4 and Lemma 5 of \cite{Dimitrov2015_1} we determined the representation of $\mathcal{L}_{h}[y]$,
\begin{equation}\label{Representation}
h^\aaaa \mathcal{L}_{h}[y]=\sum_{k=1}^{N-1}k^{1-\aaaa}\Delta_h^2  y_{N-k}+N^{1-\aaaa}\Delta_h^1 y_0,
\end{equation}
and the second-order approximation
$$\mathcal{L}_{h}[y]=h^{2-\aaaa}\sum_{k=1}^{N-1}k^{1-\aaaa}  y^{\prime\prime}_{N-k}+x^{1-\aaaa} y^\prime_{0.5}+O\llll(h^2\rrrr).$$
From Taylor expansion of the function $y(x)$ at the points $x=x_{0.5}=0.5h$ and $x=x_n=nh$ we obtain
\begin{equation}\label{Differences} 
\Delta_h^1 y_0=y_1-y_0=h y_{0.5}^\prime+\dddd{h^3}{24} y_{0.5}^{\prime\prime\prime}+O\llll(h^5\rrrr), \Delta_h^2 y_n=h^2 y_{n}^{\prime\prime}+\dddd{h^4}{12} y_{n}^{(4)}+O\llll(h^6\rrrr)
\end{equation}
\begin{lem} Let $y(t)$ be a sufficiently differentiable function. Then
$$\mathcal{L}_{h}[y]=h^{2-\aaaa}\sum_{k=1}^{N-1}k^{1-\aaaa}  y^{\prime\prime}_{N-k}+\dddd{h^{4-\aaaa}}{12}\sum_{k=1}^{N-1}k^{1-\aaaa}  y^{(4)}_{N-k}+x^{1-\aaaa} \llll(y^\prime_{0.5}+\dddd{h^2}{24}y^{\prime\prime\prime}_{0.5}\rrrr)+O\llll(h^4\rrrr).$$
\end{lem}
\begin{proof} From \eqref{Representation}  and approximations \eqref{Differences} for $\Delta_h^1 y_0$ and $\Delta_h^2 y_n$
\begin{align*}
h^\aaaa \mathcal{L}_{h}[y]=\sum_{k=1}^{N-1}k^{1-\aaaa}\bigg(h^2 y_{N-k}^{\prime\prime}+\dddd{h^4}{12}&y_{N-k}^{(4)}+O\llll(h^6\rrrr)\bigg)\\
&+n^{1-\aaaa}\llll( h y_{0.5}^\prime+\dddd{h^3}{24} y_{0.5}^{\prime\prime\prime}+O\llll(h^5\rrrr)\rrrr),
\end{align*}
\begin{align*}
\mathcal{L}_{h}[y]=h^{2-\aaaa}\sum_{k=1}^{N-1}k^{1-\aaaa} y_{N-k}^{\prime\prime}+\dddd{h^{4-\aaaa}}{12} &\sum_{k=1}^{N-1}k^{1-\aaaa} y_{N-k}^{(4)}+O\llll(h^{6-\aaaa}\rrrr)\sum_{k=1}^{n-1}k^{1-\aaaa}\\
&+x^{1-\aaaa}\llll( y_{0.5}^\prime+\dddd{h^2}{24} y_{0.5}^{\prime\prime\prime}\rrrr)+x^{1-\aaaa}O\llll(h^4\rrrr).
\end{align*}
From \eqref{SPApprox}
$$ O\llll(h^{6-\aaaa}\rrrr)\sum_{k=1}^{N-1}k^{1-\aaaa}=O\llll(h^{6-\aaaa}\rrrr)O\llll(N^{2-\aaaa}\rrrr)=O\llll(x^{2-\aaaa}h^{4}\rrrr)=O\llll(h^{4}\rrrr).
$$
Then
\begin{align*}
\mathcal{L}_{h}[y]=h^{2-\aaaa}\sum_{k=1}^{N-1}k^{1-\aaaa}  y^{\prime\prime}_{N-k}+\dddd{h^{4-\aaaa}}{12}&\sum_{k=1}^{N-1}k^{1-\aaaa}  y^{(4)}_{N-k}\\
&+x^{1-\aaaa} \llll(y^\prime_{0.5}+\dddd{h^2}{24}y^{\prime\prime\prime}_{0.5}\rrrr)+O\llll(h^4\rrrr).
\end{align*}
\end{proof}
The fractional differentiation and integration operations on a function are non-commutative. We  often consider the fractional integral as a fractional derivative of negative order. In Lemma 8 and Lemma 9 we determine the relations for the composition of the fractional derivative/integral and the second derivative.
\begin{lem} Let $y(t)$ be a sufficiently differentiable function. Then
\begin{equation}\label{SDFI}
\dddd{d^2}{dx^2}J^\aaaa y(x)=J^\aaaa y''(x)+\aaaa x^{\aaaa-1}y(0)+x^\aaaa y'(0).
\end{equation}
\end{lem}
\begin{proof} By differentiating the fractional integral $J^\aaaa y(x)$ with respect to $x$ we obtain
$$\dddd{d}{dx}J^\aaaa y(x)=\dddd{d}{dx}\int_0^x(x-t)^\aaaa y(t)dt=\aaaa \int_0^x(x-t)^{\aaaa-1} y(t)dt,$$
$$\dddd{d}{dx}J^\aaaa y(x)=\aaaa \int_0^x y(t)d\llll( -\dddd{(x-t)^{\aaaa}}{\aaaa} \rrrr)=\llll[ -(x-t)^{\aaaa}y(t) \rrrr]_0^x+ \int_0^x (x-t)^\aaaa y'(t)d t,$$
\begin{equation}\label{DJ}
\dddd{d}{dx}J^\aaaa y(x)=x^\aaaa y(0)+ J^\aaaa y'(x).
\end{equation}
By differentiating \eqref{DJ} with respect to $x$ we obtain
\begin{equation*}
\dddd{d^2}{dx^2}J^\aaaa y(x)=J^\aaaa y''(x)+\aaaa x^{\aaaa-1} y(0)+x^\aaaa y'(0).
\end{equation*}
\end{proof}
\begin{lem} Let $y(t)$ be a sufficiently differentiable function. Then
$$y^{[2+\aaaa]}(x)=y^{[\aaaa+2]}(x)+\dddd{y'(0)}{\GGGG(-\aaaa)x^{1+\aaaa}}+\dddd{y''(0)}{\GGGG(1-\aaaa)x^\aaaa}.$$
\end{lem}
\begin{proof} From \eqref{CDF} and \eqref{SDFI} we have that
$$y^{[2+\aaaa]}(x)=\dddd{d^2}{dx^2} y^{(\aaaa)}(x)=\dddd{1}{\GGGG(2-\aaaa)}\dddd{d^2}{dx^2}\llll( J^{1-\aaaa} y''(x)+x^{1-\aaaa}y'(0)\rrrr),$$
\begin{align*}
y^{[2+\aaaa]}(x)=\dddd{1}{\GGGG(2-\aaaa)}\bigg( J^{1-\aaaa} y^{(4)}(x)+&\dddd{1-\aaaa}{ x^{\aaaa}}y''(0)\\
&+x^{1-\aaaa} y'''(0)-\dddd{\aaaa(1-\aaaa)}{x^{1+\aaaa}}y'(0)\bigg),
\end{align*}
\begin{align*}
y^{[2+\aaaa]}(x)=\dddd{1}{\GGGG(2-\aaaa)}&\llll( J^{1-\aaaa} y^{(4)}(x)+x^{1-\aaaa} y'''(0)\rrrr)\\
&+\dddd{1-\aaaa}{ \GGGG(2-\aaaa)x^{\aaaa}}y''(0)-\dddd{\aaaa(1-\aaaa)}{\GGGG(2-\aaaa)x^{1+\aaaa}}y'(0).
\end{align*}
The first expression is the Caputo derivative \eqref{CDF} of the function $y''(x)$. Then
\begin{align*}
y^{[2+\aaaa]}(x)=y^{[\aaaa+2]}(x)+\dddd{y''(0)}{\GGGG(1-\aaaa)x^{\aaaa}}+\dddd{y'(0)}{\GGGG(-\aaaa)x^{1+\aaaa}}.
\end{align*}
\end{proof}
\begin{lem} Let $y(t)$ be a sufficiently differentiable function. Then
\begin{align}\label{l10}
\mathcal{L}_h[y]=\GGGG(2-\aaaa)y^{(\aaaa)}(x)&+\zzzz(\aaaa-1)y''(x)h^{2-\aaaa}-\zzzz(\aaaa-2)y'''(x)h^{3-\aaaa}\\
+\dddd{1}{12}&\llll( J^{1-\aaaa}y^{(4)}(x)+(1-\aaaa)x^{-\aaaa}y''(0)+x^{1-\aaaa}y'''(0) \rrrr)h^2\nonumber\\
&+\dddd{1}{2}\llll(\zzzz(\aaaa-3)+\dddd{\zzzz(\aaaa-1)}{6}\rrrr)y^{(4)}(x)h^{4-\aaaa}+O\llll(h^4\rrrr).\nonumber
\end{align}
\end{lem}
\begin{proof}  From Corollary 6 and Lemma 7
\begin{align*}
&h^{2-\aaaa}\sum_{k=0}^{N-1} k^{1-\aaaa} y''_{N-k}=J^{1-\aaaa}y(x)-\dddd{y''_0 x^{1-\aaaa}h}{2}+\dddd{(1-\aaaa) x^{-\aaaa}y''_0-x^{1-\aaaa} y'''_0}{12}h^2\\
&+\zzzz(\aaaa-1) y(x)h^{2-\aaaa}-y'(x)\zzzz(\aaaa-2)h^{3-\aaaa}+\dddd{y''(x)}{2}\zzzz(\aaaa-3)h^{4-\aaaa}+O\llll(h^4 \rrrr),
\end{align*}
\begin{align*}
&h^{2-\aaaa}\sum_{k=0}^{N-1} k^{1-\aaaa} y^{(4)}_{N-k}=J^{1-\aaaa}y^{(4)}(x)-\dddd{y^{(4)}_0 x^{1-\aaaa}h}{2}+\zzzz(\aaaa-1) y^{(4)}(x)h^{2-\aaaa}+O\llll(h^2 \rrrr),
\end{align*}
and 
$$\mathcal{L}_{h}[y]=h^{2-\aaaa}\sum_{k=1}^{N-1}k^{1-\aaaa}  y^{\prime\prime}_{N-k}+\dddd{h^{4-\aaaa}}{12}\sum_{k=1}^{N-1}k^{1-\aaaa}  y^{(4)}_{N-k}+x^{1-\aaaa} \llll(y^\prime_{0.5}+\dddd{h^2}{24}y^{\prime\prime\prime}_{0.5}\rrrr)+O\llll(h^4\rrrr).$$
Then
\begin{align}\label{l10_2}
\mathcal{L}_h[y]=J^{1-\aaaa}&y''(x)+J^{1-\aaaa}y^{(4)}(x)\dddd{h^2}{12}+\zzzz(\aaaa-1)y''(x)h^{2-\aaaa}\nonumber\\
&+\dddd{1-\aaaa}{12}x^{-\aaaa}y''(0)h^2-\zzzz(\aaaa-2)y'''(x)h^{3-\aaaa}\\
&+\dddd{1}{2}\llll(\zzzz(\aaaa-3)+\dddd{\zzzz(\aaaa-1)}{6}\rrrr)y^{(4)}(x)h^{4-\aaaa}+x^{1-\aaaa}S_h+O\llll(h^4\rrrr),\nonumber
\end{align}
where
$$S_h=y'_{0.5}+\dddd{h^2}{24}y'''_{0.5}-\dddd{h}{2}y''_{0}-\dddd{h^2}{12}y'''_{0}-\dddd{h^3}{24}y^{(4)}_{0}.$$
By Taylor Expansion at $x=0$,
$$y'_{0.5}=y'_0+\dddd{h}{2}y''_0+\dddd{h^2}{8}y'''_0+\dddd{h^3}{48}y^{(4)}_0+O\llll(h^4\rrrr),$$
$$y'''_{0.5}=y'''_0+\dddd{h}{2}y^{(4)}_0+O\llll(h^4\rrrr).$$
Then
$$y'_{0.5}+\dddd{h^2}{24}y'''_{0.5}=y'_0+\dddd{h}{2}y''_0+\dddd{h^2}{6}y'''_0+\dddd{h^3}{24}y^{(4)}_0+O\llll(h^4\rrrr),$$
$$S_h=y'_{0.5}+\dddd{h^2}{24}y'''_{0.5}-\dddd{h}{2}y''_{0}-\dddd{h^2}{12}y'''_{0}-\dddd{h^3}{24}y^{(4)}_{0},$$
\begin{equation}\label{l10_3}
S_h=y'_0+\dddd{h^2}{12}y'''_0+O\llll(h^4\rrrr).
\end{equation}
By substituting  \eqref{l10_3} in  \eqref{l10_2} we obtain
\begin{align*}
\mathcal{L}_h[y]=J^{1-\aaaa}y''(x)&+x^{1-\aaaa}y'(0)+\zzzz(\aaaa-1)y''(x)h^{2-\aaaa}-\zzzz(\aaaa-2)y'''(x)h^{3-\aaaa}\\
+\dddd{1}{12}&\llll( J^{1-\aaaa}y^{(4)}(x)+(1-\aaaa)x^{-\aaaa}y''(0)+x^{1-\aaaa}y'''(0) \rrrr)h^2\\
&+\dddd{1}{2}\llll(\zzzz(\aaaa-3)+\dddd{\zzzz(\aaaa-1)}{6}\rrrr)y^{(4)}(x)h^{4-\aaaa}+O\llll(h^4\rrrr).
\end{align*}
From \eqref{CDF} we have that $J^{1-\aaaa}y''(x)+x^{1-\aaaa}y'(0)=\GGGG(2-\aaaa)y^{(\aaaa)}(x)$.
\end{proof}
\begin{lem}   Let $y(t)$ be a sufficiently differentiable function. Then
\begin{align}\label{l11}
\mathcal{L}_h[y]=\GGGG(2-\aaaa)&\llll(y^{(\aaaa)}(x)+y^{[2+\aaaa]}(x)\dddd{h^2}{12}\rrrr)+\zzzz(\aaaa-1)y''(x)h^{2-\aaaa}\nonumber\\
&+\dddd{\aaaa(1-\aaaa)}{x^{1+\aaaa}}y'(0)\dddd{h^2}{12}-\zzzz(\aaaa-2)y'''(x)h^{3-\aaaa}\\
&+\dddd{1}{2}\llll(\zzzz(\aaaa-3)+\dddd{\zzzz(\aaaa-1)}{6}\rrrr)y^{(4)}(x)h^{4-\aaaa}+O\llll(h^4\rrrr).\nonumber
\end{align}
\end{lem}
\begin{proof} From \eqref{CDF} and Lemma 9,
$$D^\aaaa y''(x)=y^{[2+\aaaa]}(x)-\dddd{y'(0)}{\GGGG(-\aaaa)x^{1+\aaaa}}-\dddd{y''(0)}{\GGGG(1-\aaaa)x^{\aaaa}},
$$
\begin{equation}\label{l11_2}
J^{1-\aaaa}y^{(4)}(x)=\GGGG(2-\aaaa)D^\aaaa y''(x)-x^{1-\aaaa}y'''(0).
\end{equation}
From Lemma 8,
$$J^{1-\aaaa}y^{(4)}(x)=\GGGG(2-\aaaa)y^{[2+\aaaa]}(x)-\dddd{\GGGG(2-\aaaa)y'(0)}{\GGGG(-\aaaa)x^{1+\aaaa}}-\dddd{\GGGG(2-\aaaa)y''(0)}{\GGGG(1-\aaaa)x^{\aaaa}}-x^{1-\aaaa}y'''(0),$$
$$J^{1-\aaaa}y^{(4)}(x)=\GGGG(2-\aaaa)y^{[2+\aaaa]}(x)+\dddd{\aaaa(1-\aaaa)y'(0)}{x^{1+\aaaa}}-\dddd{(1-\aaaa)y''(0)}{x^{\aaaa}}-x^{1-\aaaa}y'''(0).$$
By substituting  the expression \eqref{l11_2} for $J^{1-\aaaa}y^{(4)}(x)$ in \eqref{l10} we obtain \eqref{l11}.
\end{proof}
The fourth-order expansion formula \eqref{4thL1} of the $L1$ approximation for the Caputo derivative is obtained from the expansion \eqref{l11} of $\mathcal{L}_h[y]$ by dividing by $\GGGG(2-\aaaa)$. The error and order of approximation \eqref{4thL1} for the functions $y(x)=e^x$ and $y(x)=\sin x$ are given in Table 3. In Theorem 13 we use the first  terms of expansion \eqref{4thL1} to derive the formula for the three-point compact approximation \eqref{CML1} for the Caputo derivative. 
\begin{clm} Suppose that $y$ is a sufficiently differentiable function. Then
\begin{equation*}\label{c12}
y(x)+\dddd{h^2}{12}y''(x)=\dddd{13}{12}y(x)-\dddd{1}{6}y(x-h)+\dddd{1}{12}y(x-2h)+O\llll(h^3\rrrr).
\end{equation*}
\end{clm}
\begin{proof} From Taylor expansion at the point $x$
$$y(x-h)=y(x)-h y'(x)+\dddd{h^2}{2}y''(x)+O\llll(h^3\rrrr),$$
$$y(x-2h)=y(x)-2h y'(x)+2h^2 y''(x)+O\llll(h^3\rrrr).$$
Then
$$\dddd{13}{12}y(x)-\dddd{1}{6}y(x-h)+\dddd{1}{12}y(x-2h)=y(x)+\dddd{h^2}{12}y''(x)+O\llll(h^3\rrrr).
$$
\end{proof}
\begin{thm} Suppose that the funcyion $y$ is a sufficiently differentiable function on the interval $[0,x_n]$ and $y'(0)=0$. Then
\begin{equation*}
\dddd{1}{\GGGG(2-\aaaa)h^\aaaa}\sum_{k=0}^{n} \dddddd_k^{(\aaaa)} y_{n-k}=\dddd{13}{12}y^{(\aaaa)}_n-\dddd{1}{6}y^{(\aaaa)}_{n-1}+\dddd{1}{12}y^{(\aaaa)}_{n-2}+O\llll(h^{3-\aaaa}\rrrr).
\end{equation*}
\end{thm}
\begin{proof} The $L1$ approximation has expansion of order $3-\aaaa$,
\begin{align*}
\dfrac{1}{\GGGG(2-\aaaa)h^\alpha}\sum_{k=0}^{n} \ssss_k^{(\alpha)} y_{n-k}= \dfrac{\mathcal{L}_{h}[y]}{\GGGG(2-\aaaa)}=y_n^{(\aaaa)}&+y_n^{[2+\aaaa]}\dddd{h^2}{12}\\
&+\dddd{\zzzz(\aaaa-1)}{\GGGG(2-\aaaa)}y''_n h^{2-\aaaa}+O\llll(h^{3-\aaaa}\rrrr).
\end{align*}
From Claim 12,
$$y^{(\aaaa)}_n+y^{[2+\aaaa]}_n\dddd{h^2}{12}=\dddd{13}{12}y^{(\aaaa)}_n-\dddd{1}{6}y^{(\aaaa)}_{n-1}+\dddd{1}{12}y^{(\aaaa)}_{n-2}+O\llll(h^3\rrrr).$$
The second-order backward difference approximation for the second derivative has first-order accuracy
$$y''_n=\dddd{y_n-2y_{n-1}+y_{n-2}}{h^2}+O\llll(h\rrrr).$$
Then
\begin{align*}
\dfrac{1}{\GGGG(2-\aaaa)h^\alpha}\sum_{k=0}^{n} \ssss_k^{(\alpha)} y_{n-k}&=\dddd{13}{12}y^{(\aaaa)}(x)-\dddd{1}{6}y^{(\aaaa)}(x-h)+\dddd{1}{12}y^{(\aaaa)}(x-2h)\\
&+\dddd{\zzzz(\aaaa-1)}{\GGGG(2-\aaaa)}\llll( \dddd{y_n-2y_{n-1}+y_{n-2}}{h^2} \rrrr) h^{2-\aaaa}+O\llll(h^{3-\aaaa}\rrrr),
\end{align*}
\begin{equation*}
\dddd{1}{\GGGG(2-\aaaa)h^\aaaa}\sum_{k=0}^{n-1} \dddddd_k^{(\aaaa)} y_{n-k}=\dddd{13}{12}y^{(\aaaa)}_n-\dddd{1}{6}y^{(\aaaa)}_{n-1}+\dddd{1}{12}y^{(\aaaa)}_{n-2}+O\llll(h^{3-\aaaa}\rrrr).
\end{equation*}
\end{proof}
\begin{table}[ht]
    \caption{Error and order of the expansion \eqref{4thL1} of the $L1$ approximation for the Caputo derivative  for the functions  $y(t)=e^t$ and $\aaaa=0.4,x=2$ (left) and $y(t)=\sin t,\aaaa=0.6,x=1$ (right).}
    \begin{subtable}{0.5\linewidth}
      \centering
  \begin{tabular}{l c c }
  \hline \hline
    $\boldsymbol{h}$ & $\mathbf{Error}$ & $\mathbf{Order}$  \\ 
		\hline \hline
$0.05$         &$2.21\times 10^{-6}$      &$4.059140$\\
$0.025$        &$1.34\times 10^{-7}$      &$4.042368$\\
$0.0125$       &$8.22\times 10^{-9}$     &$4.029349$\\
$0.00625$      &$5.07\times 10^{-10}$     &$4.019796$\\
$0.003125$     &$3.15\times 10^{-11}$     &$4.009793$\\
		\hline
  \end{tabular}
    \end{subtable}%
    \begin{subtable}{.5\linewidth}
      \centering
				\quad
  \begin{tabular}{ l  c  c }
    \hline \hline
    $\boldsymbol{h}$ & $\mathbf{Error}$ &$\mathbf{Order}$  \\ \hline \hline
$0.05$     & $2.26\times 10^{-8}$    & $4.046316$ \\
$0.025$    & $1.40\times 10^{-9}$    & $4.013357$ \\
$0.0125$   & $8.71\times 10^{-11}$   & $4.008679$ \\
$0.00625$  & $5.43\times 10^{-12}$   & $4.003960$ \\
$0.003125$ & $3.51\times 10^{-13}$   & $3.950424$ \\     
\hline
  \end{tabular}
    \end{subtable} 
\end{table}
\section{Numerical Experiments}
The fractional relaxation and subdiffusion equations are important fractional differential equations. Their analytical and numerical solutions have been studied  extensively  \cite{Alikhanov2015,DengLi2012,Dimitrov2014,Dimitrov2015_1,Dimitrov2015_2,GaoSunSun2015,GulsuOzturkAnapali2013,Lazarov2015_1,LinXu2007,Murio2008,Podlubny1999,ZengLiLiuTurner2015}. In this section we determine the numerical solutions of the fractional relaxation and subdiffusion equations which use the compact modified $L1$ approximation for the Caputo derivative
\begin{equation*}
\dddd{1}{\GGGG(2-\aaaa)h^\aaaa}\sum_{k=0}^n \dddddd_k^{(\aaaa)} y_{n-k}=\dddd{13}{12}y^{(\aaaa)}_n-\dddd{1}{6}y^{(\aaaa)}_{n-1}+\dddd{1}{12}y^{(\aaaa)}_{n-2}+O\llll(h^{3-\aaaa}\rrrr).
\end{equation*}
The fractional relaxation equation
\begin{equation} \label{RE}
y^{(\aaaa)}(x)+\llllll y(x)=F(x),\; y(0)=1,
\end{equation}
has the solution 
$$y(x)= E_{\aaaa}(-\llllll x^\aaaa)+\int_{0}^{x}\xi^{\aaaa-1} E_{\aaaa,\aaaa}\llll(-\llllll \xi^\aaaa\rrrr)F(x-\xi)d\xi.$$
The analytical solution of the  subdiffusion equation 
		\begin{equation*} \label{L41}
	\left\{
	\begin{array}{l l}
\dfrac{\partial^\alpha u(x,t)}{\partial t^\alpha}=\dfrac{\partial^2 u(x,t)}{\partial x^2},&  \\
u(x,0)=u_0(x),\;u(0,t)=u(\pi,t)=0,&  \\
	\end{array}
		\right . 
	\end{equation*}
is determined with the separation of variables method 
$$u(x,t)=\sum_{n=1}^\infty c_n E_\aaaa (-n^2  t^\aaaa) \sin (n x),$$
where the numbers $c_n$ are the coefficients of the Fourier sine series of the initial condition $u_0(x)$,
$$c_n=\dddd{2}{\pi}\int_0^\pi u_0(\xi) \sin (n \xi)d\xi.$$
In \cite{Dimitrov2015_1,Dimitrov2015_2} we derived the 
 numerical solutions $\{u_n\}_{n=0}^N$ and $\{v_n\}_{n=0}^N$ of the fractional relaxation equation \eqref{RE} which use the $L1$ approximation and the modified $L1$ approximation for the Caputo derivative,
\begin{equation}\label{L1_5}\tag{L1}
u_n=\dfrac{1}{\ssss_0^{(\alpha)}+\llllll h^\alpha\GGGG(2-\aaaa) } \left( \GGGG(2-\aaaa)h^\alpha F_n-\sum_{k=1}^n \ssss_k^{(\alpha)} u_{n-k} \right),
\end{equation}
\begin{equation}\label{ML1_5}\tag{ML1}
v_n=\dfrac{1}{\dddddd_0^{(\alpha)}+\llllll h^\alpha\GGGG(2-\aaaa) } \left( \GGGG(2-\aaaa)h^\alpha F_n-\sum_{k=1}^n \dddddd_k^{(\alpha)} v_{n-k} \right),
\end{equation}
and initial values  $u_0=v_0=y(0)$. We may also need to approximate the value of $y_1=y(h)$. A good choice for approximating $y_1$ is obtained from  the recurrence formula of numerical solution  \eqref{L1_5},
$$u_1=v_1=\frac{y(0)+\Gamma (2-a) h^a F(h)}{1+\llllll \Gamma (2-a) h^a }.
$$
In  \cite{Dimitrov2015_2} we discuss a method for improving the differentiability class  of the solutions of the fractional relaxation and subdiffusion equations. The method is based on a substitution using the fractional Taylor polynomials of the solution. When known in advance that $y(0)=y'(0)=0$, we may also set $u_1=v_1=0$.  This choice guarantees that the approximations $u_1$ and $v_1$ for the value of $y_1=y(h)$ have a second-order accuracy $O\llll( h^2 \rrrr)$.
\subsection{Numerical Solution of the Fractional Relaxation Equation }
Now we construct the numerical solution of the fractional relaxation equation on the interval $[0,1]$ using the compact modified $L1$ approximation for the Caputo derive
\begin{equation} \label{RE5}
y^{(\aaaa)}(x)+\llllll y(x)=F(x),
\end{equation}
Let $h=1/N$ and $x_k=k h$.
\begin{equation*}
\dddd{1}{\GGGG(2-\aaaa)h^\aaaa}\sum_{k=0}^n \dddddd_k^{(\aaaa)} y_{n-k}\approx \dddd{13}{12}y^{(\aaaa)}_n-\dddd{1}{6}y^{(\aaaa)}_{n-1}+\dddd{1}{12}y^{(\aaaa)}_{n-2}.
\end{equation*}
From \eqref{RE5} with $x=x_k$ we obtain $y^{(\aaaa)}_{k}=F_k-\llllll y_k$. Then
\begin{align*}
\dddd{1}{\GGGG(2-\aaaa)h^\aaaa}\sum_{k=0}^n \dddddd_k^{(\aaaa)} y_{n-k}\approx \dddd{13}{12}(F_n-&\llllll y_n)-\dddd{1}{6}\llll(F_{n-1}-\llllll y_{n-1}\rrrr)\\
&+\dddd{1}{12}\llll(F_{n-2}-\llllll y_{n-2}\rrrr).
\end{align*}
Let $w_n$ be an approximation for the solution $y_n$ at the point $x_n=n h$. The numbers $w_n$ are computed recursively with 
$w_0=y_0=y(0)$ and 
\begin{align*}
\dddd{1}{\GGGG(2-\aaaa)h^\aaaa}\sum_{k=0}^n \dddddd_k^{(\aaaa)} w_{n-k}= \dddd{13}{12}&(F_n-\llllll w_n)\\
&-\dddd{1}{6}\llll(F_{n-1}-\llllll w_{n-1}\rrrr)+\dddd{1}{12}\llll(F_{n-2}-\llllll w_{n-2}\rrrr),
\end{align*}
\begin{align}\label{CML1_5}
 w_n=\dddd{1}{12\dddddd_0^{(\aaaa)}+13\llllll (2-\aaaa)h^\aaaa}&\Bigg(-12\sum_{k=1}^n \dddddd_k^{(\aaaa)} w_{n-k}+ \GGGG(2-\aaaa)h^\aaaa(13F_n-\nonumber\\
-2(F_{n-1}&-\llllll w_{n-1})+  F_{n-2}-\llllll w_{n-2}) \Bigg).\nonumber\tag{CML1}
\end{align}
In Table 4 and Table 5 we compute the maximum error and order of numerical solutions  \eqref{L1_5},\eqref{ML1_5} and \eqref{CML1_5} for the fractional  relaxation equations \eqref{R1} and \eqref{R3} on the interval $[0,1]$. 

$\bullet$ The linear ordinary fractional differential  equation
\begin{equation}\label{R1}
y^{(\aaaa)}+y=x^{3-\aaaa}+\dddd{\GGGG(4-\aaaa)}{\GGGG(4-2\aaaa)}x^{3-2\aaaa},\quad y(0)=0
\end{equation}
has the solution $y(x)=x^{3-a}$.  Numerical solutions \eqref{ML1_5} and \eqref{CML1_5} are highly accurate approximations for the solution of equation \eqref{R1},  and their graphs are indistinguishable. In Figure 1 and Figure 2 we compare numerical solutions \eqref{ML1_5} and \eqref{CML1_5} with  numerical solution \eqref{L1_5}  on the interval $[0,1]$, when $h=0.1$.
\setlength{\tabcolsep}{0.5em}
{\renewcommand{\arraystretch}{1.1}
\begin{table}[ht]
	\caption{Maximum error and order of numerical solutions \eqref{L1_5},\eqref{ML1_5}  and \eqref{CML1_5} for relaxation equation \eqref{R1}  on the interval $[0,1]$, when $\aaaa=0.3$.}
	\centering
  \begin{tabular}{ l | c  c | c  c | c  c }
		\hline
		\hline
		\multirow{2}*{ $\quad \boldsymbol{h}$}  & \multicolumn{2}{c|}{{\bf L1}} & \multicolumn{2}{c|}{{\bf ML1}}  & \multicolumn{2}{c}{{\bf CML1}} \\
		\cline{2-7}  
   & $Error$ & $Order$  & $Error$ & $Order$  & $Error$ & $Order$ \\ 
		\hline \hline
$0.05$     & $0.00169$    & $1.59220$   & $0.0004066$          & $1.78691$  & $0.0003071$            & $2.7$     \\ 
$0.025$    & $0.00055$    & $1.61990$   & $0.0001101$          & $1.88486$  & $0.0000473$            & $2.7$      \\ 
$0.0125$   & $0.00018$    & $1.63900$   & $0.0000288$          & $1.93403$  & $7.3\times 10^{-6}$    & $2.7$    \\ 
$0.00625$  & $0.00006$    & $1.65275$   & $7.4\times 10^{-6}$  & $1.96103$  & $1.1\times 10^{-6}$    & $2.7$     \\ 
$0.003125$ & $0.00002$    & $1.66295$   & $1.9\times 10^{-6}$  & $1.97659$  & $1.7\times 10^{-7}$    & $2.7$     \\ 
\hline
  \end{tabular}
	\end{table}
	}	
	
	\begin{figure}[ht]
\centering
\begin{minipage}{.48\textwidth}
  \centering
  \includegraphics[width=.99\linewidth]{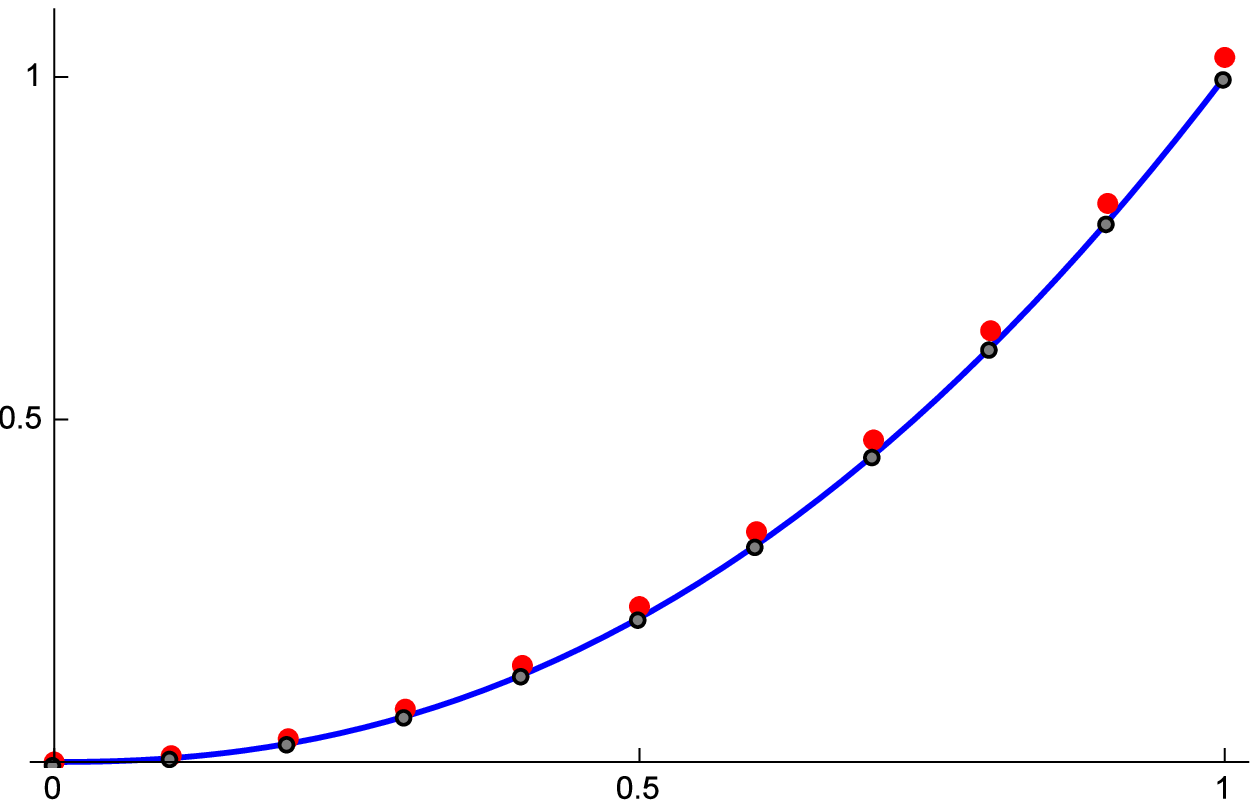}
  \captionof{figure} {Graph of the solution of equation \eqref{R1} and numerical solutions \eqref{L1_5} (red) and \eqref{ML1_5} (gray) for $\alpha=0.75$.}
\end{minipage}%
\hspace{0.2cm}
\begin{minipage}{.48\textwidth}
  \centering
  \includegraphics[width=.99\linewidth]{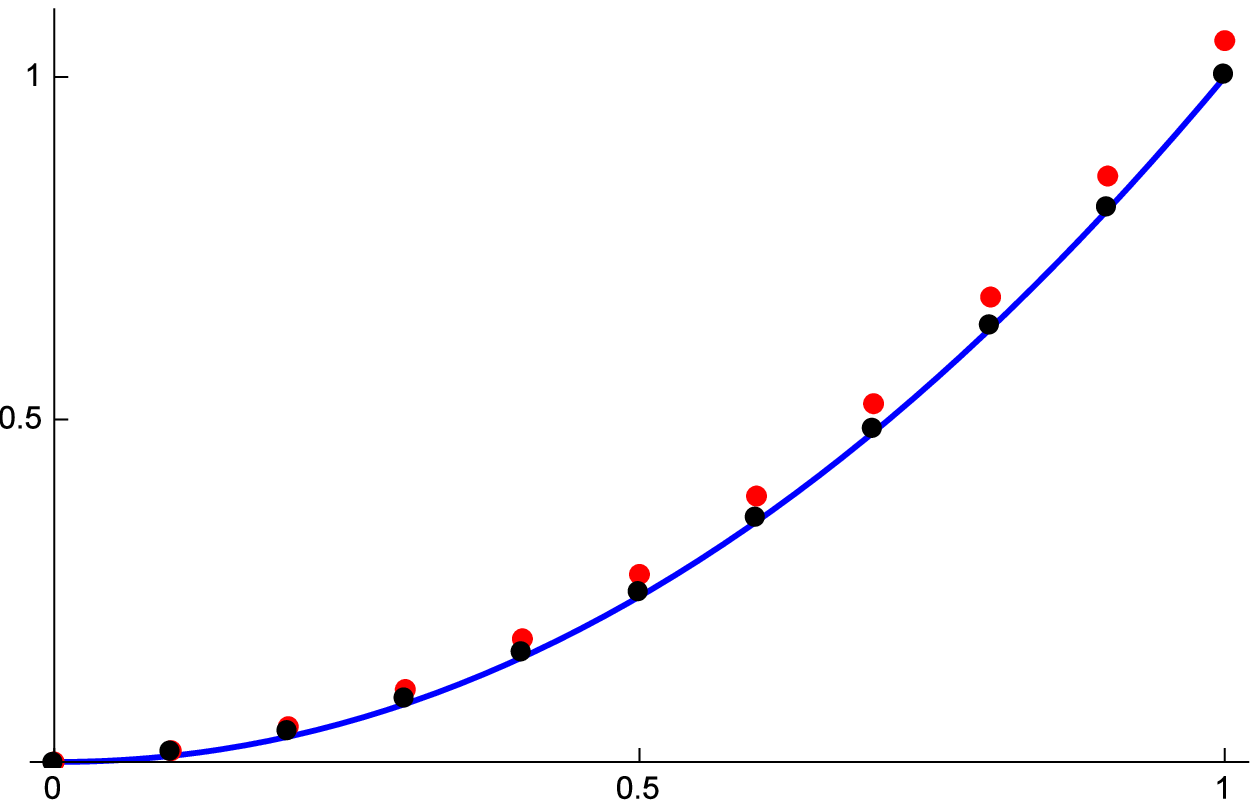}
  \captionof{figure} {Graph of the solution of equation \eqref{R1} and numerical solutions \eqref{L1_5} (red) and \eqref{CML1_5} (black) for $\alpha=0.95$.}
  \label{fig:test2}
\end{minipage}
\end{figure}

$\bullet$ The fractional relaxation equation
\begin{equation}\label{R2}
y^{(\aaaa)}+y=x^{\aaaa},\quad y(0)=1
\end{equation}
has the solution $y(x)=E_{\aaaa}\llll(-x^{a}\rrrr)+\GGGG(\aaaa+1)x^{2\aaaa}E_{\aaaa,2\aaaa+1}\llll(-x^{a}\rrrr)$. The function $E_{\aaaa}\llll(-x^{a}\rrrr)$ has an unbounded first derivative at $x=0$. The numerical solutions \eqref{L1_5} and \eqref{ML1_5} of relaxation equation \eqref{R2} have  accuracy $O\llll( h^\aaaa \rrrr)$. In \cite{Dimitrov2015_2} we described a method for improving the accuracy of the numerical solutions of the fractional relaxation and subdiffusion equations. The method uses the fractional Taylor polynomials of the solution
$$T^{(\aaaa)}_m(x)= \sum_{n=0}^m \dddd{y^{[n\aaaa]}(0)}{\GGGG(\aaaa n+1)}x^{\aaaa n}.$$
Now we determine the values of the Miller-Ross derivatives of the solution at the initial point $x=0$,
$$y^{[0]}(0)=y(0)=1,\quad y^{[\aaaa]}(0)=y^{(\aaaa)}(0)=-1.$$
By applying fractional differentiation of order $\aaaa$ to both sides of equation \eqref{R2} we obtain
$$ y^{[2\aaaa]}(x)+ y^{[\aaaa]}(x)=\GGGG(\aaaa+1),\quad y^{[2\aaaa]}(0)=1+\GGGG(\aaaa+1),$$
$$ y^{[3\aaaa]}(x)+ y^{[2\aaaa]}(x)=0,\quad y^{[3\aaaa]}(0)=-(1+\GGGG(\aaaa+1)),$$
$$ y^{[4\aaaa]}(x)+ y^{[3\aaaa]}(x)=0,\quad y^{[4\aaaa]}(0)=1+\GGGG(\aaaa+1).$$
By induction we can show that
$$y^{[n\aaaa]}(0)=(-1)^n(1+\GGGG(\aaaa+1)).$$
The solution of relaxation equation \eqref{R2} has fractional Taylor polynomials
$$T^{(\aaaa)}_m(x)=
1-\dddd{x^\aaaa}{\GGGG(\aaaa+1)}+(1+\GGGG(\aaaa+1))\sum_{n=2}^m \dddd{(-1)^n}{\GGGG(\aaaa n+1)} x^{n\aaaa}.$$
Now we apply the method from  \cite{Dimitrov2015_2}. Substitute
$$z(x)=y(x)-T^{(\aaaa)}_m(x).$$
The function $z(x)$ is a solution of the fractional relaxation equation
\begin{equation}\label{R3}
z^{(\aaaa)}+z=(-1)^{m+1}\frac{\Gamma (a+1)+1}{\Gamma (a m+1)} x^{a m},\quad z(0)=0.
\end{equation}
Equation \eqref{R3} has the solution 
\begin{align*}
z(x)=E_{\aaaa}\llll(-x^{a}\rrrr)+\GGGG(\aaaa+1)&x^{2\aaaa}E_{\aaaa,2\aaaa+1}\llll(-x^{a}\rrrr)-1\\
&+\dddd{x^\aaaa}{\GGGG(\aaaa+1)}-(1+\GGGG(\aaaa+1))\sum_{n=2}^m \dddd{(-1)^n}{\GGGG(\aaaa n+1)} x^{n\aaaa}.
\end{align*}
When $\aaaa=0.6$ and $m=4$ the solution of equation \eqref{R2} has a continuous third derivative on the interval $[0,1]$ and $z(0)=z'(0)=z''(0)=0$. In Table 5 we compute the maximum error and the order of numerical solutions \eqref{L1_5},\eqref{ML1_5} and \eqref{CML1_5} of the fractional relaxation equation \eqref{R3}   on the interval $[0,1]$. The error of numerical solution  \eqref{CML1_5} is smaller than the errors of \eqref{L1_5} and \eqref{ML1_5} when $h\leq 0.0125$. The numerical results from Table 4 and Table 5 are consistent with the expected accuracy
$O\llll( h^{2-\aaaa} \rrrr)$,$O\llll( h^2 \rrrr)$ and $O\llll( h^{3-\aaaa} \rrrr)$ of numerical solutions \eqref{L1_5},\eqref{ML1_5},\eqref{CML1_5} for the fractional relaxation equation.
\setlength{\tabcolsep}{0.5em}
{\renewcommand{\arraystretch}{1.1}
\begin{table}[ht]
	\caption{Maximum error and order of numerical solutions \eqref{L1_5},\eqref{ML1_5} and \eqref{CML1_5} of the fractional relaxation equation \eqref{R3}, when $\aaaa=0.6$ and $m=4$.}
	\centering
  \begin{tabular}{ l | c  c | c  c | c  c }
		\hline
		\hline
		\multirow{2}*{ $\quad \boldsymbol{h}$}  & \multicolumn{2}{c|}{{\bf L1}} & \multicolumn{2}{c|}{{\bf ML1}}  & \multicolumn{2}{c}{{\bf CML1}} \\
		\cline{2-7}  
   & $Error$ & $Order$  & $Error$ & $Order$  & $Error$ & $Order$ \\ 
		\hline \hline
$0.05$     & $0.00192$    & $1.33168$   & $0.0000367$          & $2.94874$  & $0.0000910$           & $2.29541$     \\ 
$0.025$    & $0.00075$    & $1.36199$   & $0.0000127$          & $1.57164$  & $0.0000179$           & $2.34716$      \\ 
$0.0125$   & $0.00029$    & $1.37706$   & $4.2\times 10^{-6}$  & $1.57238$  & $3.5\times 10^{-6}$   & $2.36921$    \\ 
$0.00625$  & $0.00011$    & $1.38564$   & $1.3\times 10^{-6}$  & $1.73716$  & $6.6\times 10^{-7}$   & $2.38225$     \\ 
$0.003125$ & $0.00004$    & $1.39085$   & $3.5\times 10^{-7}$  & $1.82637$  & $1.3\times 10^{-7}$   & $2.38973$     \\
\hline
  \end{tabular}
	\end{table}
	}
\subsection{Numerical Solutions of the Fractional Subdiffusion Equation}
	Let  $\mathcal{G}_N^M$ be a rectangular grid on  $[0,X]\times [0,T]$,
$$\mathcal{G}_N^M=\{(x_n,t_m)=(nh,m\tttt) |n=0,\cdots,N;m=0,\cdots,M\},
$$
where $h=X/N,\tttt=T/M$ and the numbers $N$ and $M$ are positive integers.
 In \cite{Dimitrov2015_1,Dimitrov2015_2} we derived the numerical solutions of the time-fractional subdiffusion equation on the rectangle $[0,X]\times [0,T]$,
			\begin{equation} \label{S5_1}
	\left\{
	\begin{array}{l l}
\dddd{\pppp^\aaaa u}{\pppp t^\aaaa}=\dddd{\pppp^2 u}{\pppp t^2}+F(x,t),\quad (x,t)\in [0,X]\times [0,T],&  \\
u(x,0)=u_0(x),\; u(0,t)=u_L(t),\; u(X,t)=u_R(t),&  \\
	\end{array}
		\right . 
	\end{equation}
 which use the $L1$ approximation \eqref{L1} and the modified $L1$ approximation \eqref{ML1} for the Caputo derivative. Let $U_n^m$ be an approximation for the value of the solution of the subdiffusion equation $u_n^m=u(n h,m \tttt)$ and
	\begin{equation}\label{ApproxR}
\dddd{1}{\GGGG(2-\aaaa)h^\aaaa}\sum_{k=0}^n \rho_k^{(\aaaa)} y_{n-k}\approx y^{(\aaaa)}_n,
\end{equation}
be an approximation for the Caputo derivative.
The finite-difference scheme for the subdiffusion equation which uses  approximation \eqref{ApproxR} is 
\begin{equation} \label{NS}
	K U^m=R^m,
	\end{equation}
	where $\eta =\GGGG(2-\alpha)\tau^{\alpha }/h^2$ and
the matrix $K$ is a tridiagonal square matrix  with values $\rho_0^{(\aaaa)}+2\eta$ on the main diagonal, and $-\eta$ on the diagonals above and below the main diagonal,
		$$K_{5} =
 \begin{pmatrix}
  \rho_0^{(\aaaa)}+2\eta &  -\eta    & 0        & 0          & 0       \\
 -\eta    &\rho_0^{(\aaaa)}+2\eta    & -\eta    & 0          & 0       \\
  0       & -\eta     & \rho_0^{(\aaaa)}+2\eta  & -\eta      & 0       \\
  0       & 0         & -\eta    & \rho_0^{(\aaaa)}+2\eta   & -\eta  \\
	0       & 0         & 0        & -\eta      &  \rho_0^{(\aaaa)}+2\eta   
 \end{pmatrix}.
$$
The vector $U^m$ is an $(N-1)$-dimensional vector whose elements $U_n^m$ are approximations of the solution on the $m^{th}$ layer of the grid $\mathcal{G}_N^M$. The vector $R^m$ has elements
$$R_1^m=\GGGG(2-\aaaa)\tttt^\aaaa F_1^m-\sum_{k=1}^m \rho_k^{(\aaaa)} U_1^{m-k}+\eta u_L(m\tttt),$$
$$R_n^m=\GGGG(2-\aaaa)\tttt^\aaaa F_n^m-\sum_{k=1}^m \rho_k^{(\aaaa)} U_n^{m-k},\qquad (2\leq n\leq N-2),$$
$$R_{N-1}^m=\GGGG(2-\aaaa)\tttt^\aaaa F_{N-1}^m-\sum_{k=1}^m \rho_k^{(\aaaa)} U_{N-1}^{m-k}+\eta u_R(m\tttt).$$
	We denote by [L1] and [ML1] the numerical solutions \eqref{NS} of the 
	subdiffusion which use the $L1$ approximation $\llll(\rho_k^{(\aaaa)}=\ssss_k^{(\aaaa)}\rrrr)$
	and the modified $L1$ approximation  $\llll(\rho_k^{(\aaaa)}=\dddddd_k^{(\aaaa)}\rrrr)$. Now we construct a  compact finite-difference scheme for numerical solution of the  subdiffusion equation using approximation \eqref{CML1} for the Caputo derivative on the three-point stencil 
	$$\{(x_n,t_m),(x_n,t_{m-1}),(x_n,t_{m-2})\},$$ 
	of the grid $\mathcal{G}_N^M$.
\begin{align*}
\dddd{1}{\GGGG(2-\aaaa)\tttt^\aaaa}\sum_{k=0}^m \dddddd_k^{(\aaaa)}  u(x_n,t_{m-k})=\dddd{13}{12}&\dddd{\pppp^\aaaa u(x_n,t_m)}{\pppp t^\aaaa}-\dddd{1}{6}\dddd{\pppp^\aaaa u(x_n,t_{m-1})}{\pppp t^\aaaa}\\
&+\dddd{1}{12}\dddd{\pppp^\aaaa u(x_n,t_{m-2})}{\pppp t^\aaaa}+O\llll(\tttt^{3-\aaaa}\rrrr).
\end{align*}
From equation \eqref{S5_1} we have
\begin{align*}
\dddd{\pppp^\aaaa u_n^{k}}{\pppp t^\aaaa}=\dddd{\pppp^2 u_n^{k}}{\pppp x^2}+F_n^{k}.
\end{align*}
Then
\begin{align*}
\dddd{1}{\GGGG(2-\aaaa)h^\aaaa}\sum_{k=0}^m \dddddd_k^{(\aaaa)}  u_n^{m-k}=\dddd{13}{12}&\llll(\dddd{\pppp^2 u_n^m}{\pppp x^2}+F_n^m\rrrr)-\dddd{1}{6}\llll(\dddd{\pppp^2 u_n^{m-1})}{\pppp x^2}+F_n^{m-1}\rrrr)\\
&+\dddd{1}{12}\llll(\dddd{\pppp^\aaaa u_n^{m-2}}{\pppp x^2}+F_n^{m-2}\rrrr)+O\llll(\tttt^{3-\aaaa}\rrrr).
\end{align*}
Denote
$$G_n^m=\dddd{13}{12}F_n^m-\dddd{1}{6}F_n^{m-1}+\dddd{1}{12}F_n^{m-2}.$$
By approximating the second derivatives in the space direction using the second-order central difference approximation we obtain
\begin{align*}
\dddddd&_0^{(\aaaa)} u_n^{m}+\sum_{k=1}^m \dddddd_k^{(\aaaa)}  u_n^{m-k}=\dddd{13\eta}{12}\llll(u_{n-1}^{m}-2u_n^{m}+u_{n+1}^{m}\rrrr)+\GGGG(2-\aaaa)h^\aaaa G_n^{m}\\
&-\dddd{\eta}{6}\llll(u_{n-1}^{m-1}-2u_n^{m-1}+u_{n+1}^{m-1}\rrrr)+\dddd{\eta}{12}\llll( u_{n-1}^{m-2}-2u_n^{m-2}+u_{n+1}^{m-2}\rrrr)+O\llll(\tttt^{3-\aaaa}+h^2\rrrr).
\end{align*}
The numbers $U_n^m\approx v_n^m$ satisfy the equations
\begin{align*}
-\dddd{13\eta}{12}U_{n-1}^{m}+\Bigg(\dddddd_0^{(\aaaa)}+\dddd{13\eta}{6}&\Bigg)U_n^{m}-\dddd{13\eta}{12}U_{n+1}^{m}=\GGGG(2-\aaaa)\tttt^\aaaa G_n^{m} \\
-\sum_{k=1}^m \dddddd_k^{(\aaaa)}&U_n^{m-k}-\dddd{\eta}{6}\llll(U_{n-1}^{m-1}-2U_n^{m-1}+U_{n+1}^{m-1}\rrrr)\\
&+\dddd{\eta}{12}\llll( U_{n-1}^{m-2}-2U_n^{m-2}+U_{n+1}^{m-2}\rrrr),
\end{align*}
where $1\leq n\leq N-1$ and  $1\leq m\leq M$. By considering the boundary conditions we obtain the three-point compact finite-difference scheme for the time fractional subdiffusion equation
\begin{equation}\label{NS5}
	K U^m=R^m,
	\end{equation}
where	$K$ is a triangular $(N-1)\times(N-1)$ matrix with values $\dddddd_0^{(\aaaa)}+13\eta/6$ on the main diagonal, and $-13\eta/12$ on the diagonals above and below the main diagonal,
		$$K_{5} =
 \begin{pmatrix}
  \dddddd_0^{(\aaaa)}+13\eta/6 &  -13\eta/12    & 0        & 0          & 0       \\
 -13\eta/12    &\dddddd_0^{(\aaaa)}+13\eta/6    & -13\eta/12    & 0          & 0       \\
  0       & -13\eta/12     & \dddddd_0^{(\aaaa)}+13\eta/6  & -13\eta/12      & 0       \\
  0       & 0         & -13\eta/12    & \dddddd_0^{(\aaaa)}+13\eta/6   & -13\eta/12  \\
	0       & 0         & 0        & -13\eta/12      &  \dddddd_0^{(\aaaa)}+13\eta/6   
 \end{pmatrix}.$$
The vector $R^m$ is an $(N-1)$-dimensional vector with elements
\begin{align*}
R_1^m=\GGGG(&2-\aaaa)\tttt^\aaaa G_1^m-\sum_{k=1}^m \dddddd_k^{(\aaaa)} U_1^{m-k}+\dddd{13}{12}u_L(m\tttt)-\dddd{1}{6}u_L((m-1)\tttt)\\
&+\dddd{1}{12}u_L((m-2)\tttt)-\dddd{\eta}{6}\llll(-2U_{1}^{m-1}+U_{2}^{m-1} \rrrr)+\dddd{\eta}{12}\llll(-2U_{1}^{m-2}+U_{2}^{m-2} \rrrr),
\end{align*}
and
\begin{align*}
R_n^m=\GGGG(2-\aaaa)&\tttt^\aaaa G_n^m-\sum_{k=1}^m \dddddd_k^{(\aaaa)} U_n^{m-k}\\
&-\dddd{\eta}{6}\llll(U_{n-1}^{m-1}-2U_{n}^{m-1}+U_{n+1}^{m-1} \rrrr)+\dddd{\eta}{12}\llll(U_{n-1}^{m-2}-2U_{n}^{m-2}+U_{n+1}^{m-2} \rrrr),
\end{align*}
for $2\leq n\leq N-2$
\begin{align*}
R_{N-1}^m=&\GGGG(2-\aaaa)\tttt^\aaaa G_{N-1}^m-\sum_{k=1}^m \dddddd_k^{(\aaaa)} U_{N-1}^{m-k}+\dddd{13}{12}u_R(m\tttt)-\dddd{1}{6}u_R((m-1)\tttt)\\
&+\dddd{1}{12}u_R((m-2)\tttt)-\dddd{\eta}{6}\llll(U_{N-2}^{m-1}-2U_{1}^{m-1}\rrrr)+\dddd{\eta}{12}\llll(U_{N-2}^{m-2}-2U_{N-1}^{m-2} \rrrr).
\end{align*}
Denote by [CML1] the three-point compact finite-difference scheme \eqref{NS5} for the time-fractional subdiffusion equation. In Table 6 and Table 7 we compute the order of numerical solutions [L1], [ML1], [CML1] in the directions of space and time. The orders are computed by fixing the value of one of the parameters $N=50$ or $M=50$ and computing the order of the numerical solution by doubling the value of the other parameter.

$\bullet$ The fractional subdiffusion equation
	\begin{equation}\label{S05}
	\left\{
	\begin{array}{l l}
\dddd{\pppp^\aaaa u}{\pppp t^\aaaa}=\dddd{\pppp^2 u}{\pppp t^2}+F(x,t)\quad (x,t)\in [0,1]\times [0,1],&  \\
u(x,0)=0,\quad u(0,t)=u_L(t)=t^{3-\aaaa},\quad u(1,t)=u_R(t)=6 t^{3-\aaaa},&  \\
	\end{array} 
		\right . 
	\end{equation}
	where
		$$F(x,t)=\dddd{\GGGG(4-\aaaa)}{\GGGG(4-2\aaaa)}\llll(1+2x^2+3x^3\rrrr)t^{3-2\aaaa}-(4-18x)t^{3-\aaaa},
	$$
	has the solution $u(x,t)=\llll(1+2x^2+3x^3\rrrr)t^{3-\aaaa}$. The solution  of equation \eqref{S05}  satisfies the condition $y'(0)=0$.
\setlength{\tabcolsep}{0.5em}
{\renewcommand{\arraystretch}{1.1}
\begin{table}[ht]
	\caption{Space and time  orders of numerical solutions [L1], [ML1] and [CML1] for equation \eqref{S05} on the interval $[0,1]$, when $\aaaa=0.4$.}
	\centering
		  \begin{tabular}{ l | c  c | c  c | c  c }
			\hline
		\hline
		\multirow{2}*{ $\quad  \boldsymbol{h/\tttt}$}  & \multicolumn{2}{c|}{{\bf L1}} & \multicolumn{2}{c|}{{\bf ML1}}  & \multicolumn{2}{c}{{\bf CML1}} \\
		\cline{2-7}  
   & $Space$ & $Time$  &  $Space$ & $Time$  &  $Space$ & $Time$\\
		\hline \hline
 0.0125& $1.99957$    & $1.52246$   & $1.99969$   & $1.95593$  & $1.9989$    & $2.16901$    \\ 
0.00625 & $1.99990$    & $1.55257$   & $1.99951$   & $1.93991$  & $1.99972$   & $2.42599$    \\ 
0.003125&  $2.00107$    & $1.56733$   & $2.01293$   & $1.95242$  & $1.99993$   & $2.51445$    \\ 
0.0015625&  $1.99926$    & $1.57640$   & $2.09988$   & $1.96646$  & $1.99995$   & $2.55358$    \\ 
0.00078125&  $2.01012$    & $1.58260$   & $2.16357$   & $1.97727$  & $2.00010$   & $2.57319$    \\
\hline
  \end{tabular}
	\end{table}
	}
	
	$\bullet$ The fractional subdiffusion equation
		\begin{equation}\label{S06}
	\left\{
	\begin{array}{l l}
\dddd{\pppp^\aaaa u}{\pppp t^\aaaa}=\dddd{\pppp^2 u}{\pppp t^2}\quad (x,t)\in [0,\pi]\times [0,\pi],&  \\
u(x,0)=\sin x,\quad u(0,t)=0,\quad u(\pi,t)=0,&  \\
	\end{array} 
		\right . 
	\end{equation}
	has the solution $u(x,t)=\sin x E_\aaaa\llll(-t^\aaaa \rrrr)$. In \cite{Dimitrov2015_2} we determined  the partial Miller-Ross derivatives of the solution at $t=0$,
		$$\llll. \dddd{\pppp^{n\aaaa} u(x,t)}{\pppp t^{n\aaaa}}\rrrr|_{t=0}
		=(-1)^{n}\sin x.$$
The  solution of \eqref{S06} has fractional Taylor polynomials of  at $t=0$
		$$T_m^{(\aaaa)}(x,t)= \sin x \sum_{n=0}^m (-1)^n \dddd{t^{n\aaaa}}{\GGGG(n\aaaa+1)}.$$
Substitute
		$$v(x,t)=u(x,t)-\sin x \sum_{n=0}^m (-1)^n \dddd{t^{n\aaaa}}{\GGGG(n\aaaa+1)}.$$
	The function $v(x,t)$ is a solution of the  subdiffusion equation \cite{Dimitrov2015_2}
		\begin{equation*}
	\left\{
	\begin{array}{l l}
	\dfrac{\partial^{\aaaa} v(x,t)}{\partial t^{\aaaa}}=\dfrac{\partial^2 v(x,t)}{\partial x^2}+(-1)^{m+1}\sin x\dddd{t^{m\aaaa}}{\GGGG(m\aaaa+1)},&  \\
	v(x,0)=v(0,t)=v(\pi,t)=0.&  \\
	\end{array}
		\right . 
	\end{equation*}
When $\aaaa=0.5,m=4$  we obtain the    equation
	\begin{equation}\label{S15}
	\left\{
	\begin{array}{l l}
	\dfrac{\partial^{0.5} u(x,t)}{\partial t^{0.5}}=\dfrac{\partial^2 u(x,t)}{\partial x^2}-\dddd{t^{2}\sin x}{2},&  \\
	u(x,0)=\sin x,\;u(0,t)=u(\pi,t)=0,&  \\
	\end{array} 
		\right . 
	\end{equation}
	When   $\aaaa=0.5,m=5$ we obtain the equation
	\begin{equation}\label{S16}
	\left\{
	\begin{array}{l l}
	\dfrac{\partial^{0.5} u(x,t)}{\partial t^{0.5}}=\dfrac{\partial^2 u(x,t)}{\partial x^2}+\dddd{15\sqrt \pi t^{2.5} \sin x }{8},&  \\
	u(x,0)=\sin x,\;u(0,t)=u(\pi,t)=0,&  \\
	\end{array} 
		\right . 
	\end{equation}
	In Table 7 we compute the  orders of numerical solutions [L1], [ML1], [CML1] for the subdiffusion equations \eqref{S15} and \eqref{S16} in the space and time directions. 	
\setlength{\tabcolsep}{0.5em}
{\renewcommand{\arraystretch}{1.1}
\begin{table}[ht]
	\caption{Space and time  orders of numerical solutions [L1], [ML1]   for equation \eqref{S15} and numerical solution [CML1] for equation \eqref{S16}}
	\centering
		  \begin{tabular}{ l | c  c | c  c | c  c }
		\hline
		\hline
		\multirow{2}*{ $\quad \boldsymbol{h/\tttt}$}  & \multicolumn{2}{c|}{{\bf L1}} & \multicolumn{2}{c|}{{\bf ML1}}  & \multicolumn{2}{c}{{\bf CML1}} \\
		\cline{2-7}  
   & $Space$ & $Time$  &  $Space$ & $Time$  &  $Space$ & $Time$\\
		\hline \hline
 0.03926991&$2.00023$    & $1.3777$    & $2.00015$   & $2.43060$  & $2.00022$   & $2.14451$   \\ 
  0.01963495&$2.00006$    & $1.44914$   & $2.00004$   & $2.22449$  & $2.00006$   & $2.35878$   \\ 
  0.00981748&$2.00001$    & $1.47435$   & $2.00001$   & $2.10396$  & $2.00001$   & $2.43433$   \\ 
  0.00490874&$2.00000$    & $1.48522$   & $2.00000$   & $2.04184$  & $2.00000$   & $2.46707$   \\ 
  0.00245437&$2.00000$    & $1.49071$   & $2.00000$   & $2.01288$  & $2.00001$   & $2.48269$   \\
\hline
  \end{tabular}
	\end{table}
	}
	\section{Conclusions}
	In the present paper we obtained the  compact modified $L1$ approximation \eqref{CML1} for the Caputo fractional derivative and we computed the numerical solutions of the fractional relaxation and  subdiffusion equations. 
	 Numerical solution \eqref{NS5} is the first three-point compact finite-difference scheme for the fractional subdiffusion equation which uses the coefficients of the $L1$ approximation and the value of the Riemann zeta function at the point $\aaaa-1$. The accuracy of the compact modified $L1$ approximation  is $O\llll(h^{3-\aaaa}\rrrr)$, while the three-point compact Gr\"nwald approximation \eqref{3PTG} has accuracy $O\llll(h^{3}\rrrr)$. One advantage of compact approximation \eqref{CML1} is that the coefficients on the right-hand side are the same for all values of $\aaaa$, as well as the that the only requirement for the function $y$ is that $y'(0)=0$.
	An important question for future work is to prove the convergence of the numerical methods discussed in the paper. Another question for future work is to derive higher-order compact approximations based on the the $L1$ approximation and Gr\"unwald formula  approximation for the Caputo derivative.

\end{document}